\documentclass[11pt]{article}
\usepackage[centering, margin={0.8in, 0.8in}, includeheadfoot]{geometry}
\usepackage{tikz,enumerate}
\usetikzlibrary{arrows,shapes,positioning}
\usetikzlibrary{decorations.markings}
\tikzstyle arrowstyle=[scale=1]
\tikzstyle directed=[postaction={decorate,decoration={markings,
    mark=at position .65 with {\arrow[arrowstyle]{stealth}}}}]

\usepackage{amsthm,amsfonts,amssymb,epsfig,graphics,amsbsy,amsmath,amssymb,float,subfloat,epstopdf,subfig}
\usepackage{enumitem}
\newcommand{\half}{^\infty_0}

\newcommand{\RR}{\mathbb{R}}
\newcommand{\xx}{\mathbf{x}}
\newcommand{\yy}{\mathbf{y}}

\newcommand{\intL}{\int\limits }
\newtheorem{thm}{Theorem}

\newtheorem{prop}[thm]{Proposition}
\newtheorem{algorithms}[thm]{Algorithm}
\newtheorem{cor}[thm]{Corollary}
\newtheorem{lem}[thm]{Lemma}
\usetikzlibrary{calc}
\begin{document}

\title{Photoacoustic tomography with time-dependent damping: Theoretical and a convolutional neural network-guided numerical inversion procedure}
\author{Sunghwan Moon$^{a}$, Anwesa Dey$^{b}$,  and Souvik Roy$^{c}$}
\date{ ${ }^a$ Department  of  Mathematics,  Kyungpook  National  University,  Daegu  41566,  Republic  of  Korea\\
${ }^b$ Department  of  Mathematics,  University of Utah, Salt Lake City, UT 84112,  USA\\
${ }^c$ Department  of  Mathematics,  The University of Texas at Arlington,  Arlington,  TX 76019,  USA\\}
\maketitle
\begin{abstract}
In photoacoustic tomography (PAT), a hybrid imaging modality that is based on the acoustic detection of optical absorption from biological tissue exposed to a pulsed laser, a short pulse laser generates an initial pressure proportional to the absorbed optical energy, which then propagates acoustically and is measured on the boundary. 
To account for the significant signal distortion caused by acoustic attenuation in biological tissue, we model PAT in heterogeneous media using a damped wave equation featuring spatially varying sound speed and a time-dependent damping term.
Under natural assumptions, we show that the initial pressure is uniquely determined by the boundary measurements using a harmonic extension of the boundary data with energy decay.
For constant damping, an expansion in Dirichlet eigenfunctions of $-c^2(\xx)\Delta$ leads to an explicit series reconstruction formula for the initial pressure. Finally, we develop a gradient free numerical method based on the Pontryagin's maximum principle to provide a robust and computationally viable approach to image reconstruction in attenuating PAT.
\end{abstract}

\textbf{keywords:} Wave equation, time-dependent damping, photoacosutic tomography, optimization, inversion

\textbf{MSCcodes:} 35R30, 35L05, 49J20, 49K20

\section{Introduction}\label{sec:intro}

Optical tomography (OT) is a biomedical imaging modality that exploits the propagation of light in the near–infrared (NIR) spectral range to reconstruct the optical properties of biological media, such as tissues. Since malignant tissue typically scatters light less than healthy tissue, accurate recovery of spatially varying scattering and absorption coefficients is essential for reliable cancer diagnosis. However, the inverse problem in OT is severely ill–posed and, in practice, often leads to reconstructions with poor spatial resolution. To overcome these limitations and obtain high–fidelity maps of optical parameters, a variety of multi–physics imaging techniques, such as fluorescence molecular tomography, confocal diffuse tomography, single–photon photoacoustic tomography, two–photon photoacoustic tomography, and ultrasound–modulated optical tomography, have been developed under the umbrella of hybrid OT methods (see, e.g., \cite{ammaribgns14,dey2024,Gup2021}). Among these, photoacoustic tomography (PAT) \cite{Gao12,scherzer15,wang09} has emerged as one of the most prominent approaches owing to its wide range of biomedical applications. Consequently, the development of accurate and efficient computational models for image reconstruction in PAT is of considerable importance.

In PAT, a heterogeneous object, typically biological tissue, is illuminated along its boundary by short pulses of NIR light. As photons propagate through the medium, a fraction of the optical energy is absorbed, causing localized heating and subsequent thermoelastic expansion of the tissue, while the remaining photons undergo multiple scattering events. After the illumination ceases, the tissue cools and contracts. This sequence of rapid expansion and contraction generates transient pressure variations inside the object, which in turn produce acoustic (ultrasound) waves through the photoacoustic effect \cite{bell80}. These pressure waves propagate outward and are recorded by ultrasound detectors positioned on (or near) the surface of the object. Acoustic propagation in tissue occurs on microsecond time scales, which are roughly three orders of magnitude slower than optical diffusion. The core mathematical task in PAT is an inverse problem: reconstructing the initial pressure distribution, which is proportional to the absorbed optical energy density, from the measured acoustic signals at the boundary (see, for instance, \cite{ammaribjk10,anastasiozmr07,kuchment14book,moonzangerlip19} and the references therein).

Many mathematical studies on PAT model the pressure field by the standard wave equation with a variable sound speed; see, for instance, \cite{moonjop16,hwangjmp24,knoxm20,liuu15,moonsiims23} and the references therein. 
Although such models capture the finite–speed propagation and heterogeneity of sound speed reasonably well, they neglect the attenuation mechanisms that are intrinsic to biological tissue. 
In reality, as pressure waves travel through the medium, their amplitudes are gradually reduced due to absorption, scattering, and viscous losses. 
This attenuation can significantly influence the measured signals and, if not appropriately accounted for, may lead to systematic reconstruction artifacts and loss of resolution. To address this issue, several attenuation models have been proposed in the PAT literature \cite{Ammari2012,elbauss17,moonip24,kowars12,nachmansw90}, including frequency–power–law models, thermo–viscous models, and more general convolution–in–time formulations; the corresponding inverse problems are typically more challenging and may suffer from increased ill–posedness. 
Moreover, most of these models are formulated under the simplifying assumption of a constant sound speed.

In this work, we focus on a time–dependent damped wave equation in PAT. 
Specifically, we consider the pressure field \(p(\xx,t)\) as the solution to the Cauchy problem
\begin{equation}
\partial_t^2 p(\xx,t) + \gamma(t)\,\partial_t p(\xx,t) - c(\xx)^2 \Delta p(\xx,t) = 0,
\qquad (\xx,t)\in \mathbb{R}^n\times[0,\infty),
\label{eq:attwaveorigin}
\end{equation}
where $n\in\mathbb N$ denotes the spatial dimension, with the initial data
\begin{equation}\label{eq:attwave-ic}
p(\xx,0)=p_0(\xx), \qquad \partial_t p(\xx,0)=0.
\end{equation}
Here, \(c(\cdot)\) denotes a spatially varying sound speed and \(\gamma(\cdot) > 0\) is a smooth, strictly positive damping function that models cumulative attenuation phenomena such as absorption and viscous losses. 
We assume that  the region of interest \(\Omega\subset\mathbb{R}^n\) is a bounded domain with smooth boundary \(\partial\Omega\) and that the sound speed, the damping coefficient, and the initial data satisfy the assumptions~(\ref{eq:assumptions}) stated in Section~\ref{sec:analysis}.
The measured data are given by the boundary trace
\[
g(\yy,t) = p(\yy,t)\big|_{(\yy,t)\in \partial\Omega\times[0,\infty)}.
\]
In this paper, we study how to reconstruct the initial pressure distribution \(p_0\), and hence the absorbed optical energy density, inside \(\Omega\) from the knowledge of \(g\).

We also note that, when the damping coefficient in \eqref{eq:attwaveorigin} depends only on space, $\gamma=\gamma(\xx)$, related inverse problems have been studied in the literature \cite{dohknn24,haltmeiern19,homan13,palacios16}. 
For variable sound speed $c(\cdot)$ the inverse problem can be treated by
time–reversal based reconstruction/the Neumann series reconstruction \cite{hristovakn08,nguyenhkd22,oksanenu14,stefanovu09}, and in \cite{haltmeiern19} this approach
was further extended to the case of spatially dependent damping. 
In the context of multiwave imaging, Homan~\cite{homan13}
 and Palacios~\cite{palacios16} analyzed damped wave equations with spatially varying attenuation and derived Neumann-series-type reconstruction formulas under suitable assumptions.
However, in
the presence of time–dependent damping the underlying time–reversal symmetry
is lost, and such time–reversal techniques no longer apply in a straightforward
way.

Our first contribution is an analytical study of the inverse problem for the time–dependent damped wave equation \eqref{eq:attwaveorigin}. 
Under assumptions of natural regularity and boundedness in \(c\) and \(\gamma\), we construct an auxiliary field \(u\) satisfying homogeneous Dirichlet conditions on \(\partial\Omega\) whose energy decays in time. 
Following the eigenfunction–expansion framework based on harmonic extension and Dirichlet eigenfunctions that was originally employed in the PAT context by Agranovsky and Kuchment~\cite{agranovskyk07}, we introduce the harmonic extension operator \(E\) from \(\partial\Omega\) into \(\Omega\) and expand the unknown field in terms of Dirichlet eigenfunctions \(\{\psi_k\}_{k\in\mathbb{N}}\) of the elliptic operator \(-c(\xx)^2\Delta\). 
On this frame, we establish that the initial pressure distribution \(p_0\) is uniquely determined by the boundary measurements \(g\) on \(\partial\Omega\times[0,\infty)\), thus providing a uniqueness result for the PAT inverse problem in the presence of a general time–dependent damping term.

Our second contribution is an explicit inversion procedure in the special but practically relevant case where the damping coefficient is constant, \(\gamma(t)=\gamma>0\). 
Using the harmonic extension operator and the unknown field in terms of Dirichlet eigenfunctions, the damped wave equation is reduced to a family of decoupled second–order ordinary differential equations for the coefficients. For constant \(\gamma\), these ODEs can be solved explicitly, which yields closed–form expressions for the coefficients of \(u\) and, consequently, for \(p_0 - Eg(\cdot,0)\) in terms of time–convolutions of the measured data \(g\) and its time derivatives. This leads to an analytic series reconstruction formula for \(p_0\) inside \(\Omega\) in the weighted space \(L^2(\Omega,c^{-2})\).

For the numerical reconstruction of the initial pressure $p_0$, we aim to minimize a Tikhonov-type least-squares functional that incorporates data fitting terms and sparsity-promoting regularization terms.  While the inclusion of nonsmooth regularization significantly enhances the method’s ability to produce sharp reconstructions, as demonstrated in \cite{Ade2018,Roy2018}, it also necessitates advanced optimization techniques suitable for nonsmooth settings, two of which are proximal methods and  semi-smooth Newton (SSN) methods \cite{Bartsch2021,Gupta2020,Gupta2021}.
However, implementing these approaches requires evaluating the subgradients and Hessian of the optimization problem, which involves solving not only the state and adjoint equations but also two additional linearized systems. These requirements, together with other implementation challenges, render these schemes demanding for practical use.

To address these challenges, we use the sequential quadratic Hamiltonian (SQH) algorithm, which was originally proposed in \cite{Breit19,Breit18,Breit19a} for nonsmooth PDE-constrained optimal control problems, to our inverse problem. The SQH method is grounded in the Pontryagin maximum principle (PMP), a cornerstone of optimal control theory. We are particularly motivated by the SQH algorithm’s ease of implementation, strong efficiency and robustness, and well-posedness as an iterative scheme \cite{Breit18,Breit19a}. A distinctive feature of this iterative approach is its reliance on a pointwise optimization procedure. Such a scheme has been recently used in context of tomographic inverse problems arising in optical tomography \cite{dey2024,Roy2022,Roy2023,Roy2025}. However, since the SQH method is iterative, its performance depends heavily on initial guesses. Traditional options are to either provide a specified initial guess, like 0, or construct an approximate solution using some other algorithm and use that as the initial guess. It is worthwhile to note that iterative algorithms like SQH work better if the structure of the initial guess resembles the structure of the true solution. One viable approach would be to use the time-reversal algorithm and obtain an approximation of $p_0$, to be used as an initial guess. However, even though the time-reversal algorithm might be able to avoid appearance of spurious discontinuities or violate the underlying physics, due to the structure of the attenuated wave equation, there will be significant damping, resulting in huge loss of contrast.   

On the other hand, convolutional neural networks (CNNs) have proven highly effective for recovering parameters in ill-posed, PDE-constrained inverse problems. The CNNs learns directly from data. Their data-driven nature allows them to extract complex, problem-specific features that would be difficult or even impossible to encode analytically as explicit priors. Moreover, CNNs are well suited for representing nonlinear relationships: by stacking convolution operations, which are linear, with nonlinear activation functions, they build hierarchical representations capable of approximating highly nonlinear mappings between inputs and outputs. This interplay of layered convolutions and nonlinear activations enables CNNs to model intricate structures that arise in challenging inverse problems. Furthermore, unlike classical optimization-based approaches, where regularization must be introduced explicitly through penalty terms such as $L^2,L^1$, or Tikhonov regularization, CNNs embed regularization implicitly within their architecture and training dynamics. 

However, one major drawback of CNN (or in general any traditional deep learning algorithm) is the inability to incorporate the underlying physical structure of the problem and relies on training with a huge dataset to provide meaningful inference, especially for complex inverse problems that usually appear with PDEs. There have been some progress made in the field of machine learning through the development of physics-informed neural networks (PINN), but it still struggles to find a balance between learning patterns from data and satisfying physical constraints. Thus, the output of CNN, in general, can suffer from spurious artifacts. 

In an attempt to circumvent this issue, we combine the advantages of time-reversal and CNN to choose a initial guess of our SQH as the sum of the two outputs. This ensures that the initial guess preserves some information of the discontinuities present as well as provide better starting contrast. As we will note in the results in Section \ref{sec:numericalresult}, such a choice of initial guess guides the SQH algorithm to provide superior contrast and resolution reconstructions.

The remainder of the paper is organized as follows. In Section~\ref{sec:analysis}, we present the analytical framework for the damped wave equation, including well–posedness and decay estimates, and we derive the uniqueness result for the initial pressure as well as an explicit inversion formula in the case of constant damping. In Section~\ref{sec:control}, we introduce the optimization framework for solving the initial pressure numerically. The framework comprises of a cost functional with data-fitting and regularization terms. We characterize the minimizer using the PMP, and describe the SQH–based reconstruction algorithm, guided by an initial guess obtained using a CNN, for the numerical PMP implementation along with results for the convergence of the algorithm. In Section ~\ref{sec:numericalresult}, we discuss the results of the reconstruction with our proposed algorithm and compare our outputs with the time-reversal algorithm, both qualitatively and quantitatively. We conclude the paper with a section of conclusions.

\section{Analysis of the damped wave model}
\label{sec:analysis}
In this section we study the Cauchy problem \eqref{eq:attwaveorigin}–\eqref{eq:attwave-ic}
introduced in Section~\ref{sec:intro}. Throughout the paper we impose the following
standing assumptions on $\gamma$, $c$, $p_0$, and the bounded domain  $\Omega\subset \RR^n$ with smooth boundary $\partial \Omega$:  for $\nu\in\mathbb N, \nu>n/2+1$,
\begin{equation}\label{eq:assumptions}
\begin{aligned}
&\bullet\ \gamma(\cdot) \in C^{\nu+1}([0,\infty)) \ \text{with}\quad 0<\gamma_m\le \gamma(t) \le \gamma_M,\\
&\bullet\ c(\cdot) \in C^{\nu+1}(\RR^n) \ \text{with}\quad 0<c_m \le c(\xx)\le c_M,\\
&\bullet\ p_0 \in H^\nu(\RR^n) \ \text{with compact support},\\
&\bullet\ \Omega\subset \RR^n \ \text{is an open and bounded domain with smooth boundary } \partial \Omega.
\end{aligned}
\end{equation}

We notice that every scalar hyperbolic equation of second order can be transformed into a symmetric hyperbolic system (see \cite[page 21]{rackebook2015}).
Hence by Theorem 3.3 in \cite{rackebook2015},  the initial problem (\ref{eq:attwaveorigin}) with $
p(\xx,0) = p_0(\xx), \partial_t p(\xx,0) = 0
$ under assumption (\ref{eq:assumptions}) has a unique solution $p\in C^2([0,\infty)\times\mathbb R^n)$ and
$$
p\in C^1([0,\infty); H^\nu(\RR^n)\cap C^2([0,\infty); H^{\nu-1}(\RR^n)),
$$
where $H^\nu(\RR^n),\nu\in \RR$ is the Sobolev space.

Let $E$ denote the harmonic extension operator from $\partial\Omega$ to $\Omega$, that is, the operator which
produces a harmonic function $E\phi$ in $\Omega$ for given Dirichlet data $\phi$ on $\partial\Omega$.
By standard regularity theory, $E$ is continuous from $H^s(\partial\Omega)$ to $H^{s+1/2}(\Omega)$ for any $s>0$.
Hence we can rewrite \eqref{eq:attwaveorigin} in terms of $u = p - Eg$ as follows:
\begin{equation}\label{eq:attwave}
\begin{array}{ll}
\partial_t^2 u(\xx,t)+\gamma(t) \partial_t u(\xx,t)-c(\xx)^2\Delta_{} u(\xx,t)=-G_\gamma (\xx,t) \\
\qquad \mbox{for }(\xx,t)\in \Omega\times[0,\infty)\\
u(\xx,0)=p_0(\xx)-Eg(\xx,0)\quad \mbox{and}\quad  \partial_t u(\xx,0)=0\quad \mbox{for }\xx\in \Omega \\
u(\yy,t)=0\qquad (\yy,t)\in \partial\Omega\times [0,\infty),
\end{array}
\end{equation}
where $G_\gamma (\xx,t)=E(\partial_{t}^2 g)(\xx,t)+\gamma(t)E(\partial_{t} g)(\xx,t)$.

Let $\{\lambda_k^2,\psi_k\}_{k=1}^\infty$ be the eigenvalues and an associated orthonormal basis of eigenfunctions of
the operator $-c(\xx)^2\Delta$ with Dirichlet boundary conditions, where $\lambda_k>0$ for all $k$.
We use the notation
$$
f_k := \langle f,\psi_k\rangle, u_k(t):=\langle u(\cdot,t),\psi_k\rangle,\quad\mbox{and}\quad
G_{\gamma,k}(t):=\langle G_\gamma(\cdot,t),\psi_k\rangle
$$
for the Fourier coefficients of $f$ and $u$ with respect to $\{\psi_k\}$.
Here $\langle\cdot,\cdot\rangle$ denotes the inner product in $L^2(\Omega,c^{-2})$.
\begin{thm}\label{thm:unique}
Under assumptions \eqref{eq:assumptions}, the initial function $p_0(\xx)-Eg(\xx,0)$ can be uniquely determined from $Eg(\xx,t)$.
Furthermore, $p_0(\xx)$ can be uniquely determined from $g$.
\end{thm}
To establish the above uniqueness result and to prepare for an explicit reconstruction formula, 
we expand the solution in the eigenbasis of the operator $-c(\xx)^2\Delta$ and study the time evolution of the corresponding modal coefficients. 
The following lemma provides that (\ref{eq:attwave}) satisfies these coefficients.

\begin{lem}\label{lem:t}
If for any $k\in\mathbb N$, $u_k$ satisfies 
\begin{equation}\label{eq:attwavet1}
u_k''(t)+\gamma(t)u_k'(t)+\lambda_k^2u_k(t)=-G_{\gamma,k}(t),
\end{equation}
then $L^2(\Omega, c^{-2})$-convergent series  $u=\sum u_k\psi_k$ satisfies the PDE in (\ref{eq:attwave}).
\end{lem}

\begin{proof}
For any $k\in \mathbb N$, let us consider 
$$
\begin{array}{ll}
\displaystyle\intL_{\Omega}\Delta_{\xx}u(\xx,t)\overline{\psi_k(\xx)}{\rm d}\xx&=\displaystyle\intL_{\Omega}u(\xx,t)\Delta_{\xx}\overline{\psi_k(\xx)}{\rm d}\xx+\intL_{\partial \Omega}\partial_\nu u(\yy,t) \overline{\psi_k(\yy)}-u(\yy,t) \partial_\nu\overline{\psi_k(\yy)}{\rm d}s(\yy)\\
&=-\displaystyle\intL_{\Omega}u(\xx,t)\lambda_k^2 \overline{\psi_k(\xx)}c(\xx)^{-2}{\rm d}\xx,
\end{array}
$$
where  $\partial_\nu$ is the normal derivative on $\partial \Omega$, ${\rm d}s(\yy)$ denotes the surface measure on $\partial \Omega$ and $\overline{z},z\in \mathbb C,$ is the complex conjugate of $z$.
Here in the first equality, we used the Green's theorem 
and we used the fact that $u$ and $\psi_k$ in $\partial \Omega$ are equal to $0$ in the second equality.
Multiplying the PDE in (\ref{eq:attwave}) by $\psi_k$ and integrating over $\Omega$ with weight $c(\xx)^{-2}$, we have
\begin{equation*}\label{eq:attwavet}
\begin{array}{ll}
 \displaystyle-G_{\gamma,k}(t)
&=\displaystyle\intL_{\Omega}\{(\partial_t^2+\gamma(t)  \partial_t-c(\xx)^2\Delta_{\xx}) u(\xx,t)\}\overline{\psi_k(\xx)}c(\xx)^{-2}{\rm d}\xx\\
&=\displaystyle\intL_{\Omega}\{(\partial_t^2 +\gamma(t)  \partial_t+\lambda_k^2) u(\xx,t)\}\overline{\psi_k(\xx)}c(\xx)^{-2}{\rm d}\xx.
\end{array}
\end{equation*}
Hence if for any $k\in\mathbb N$ $u_k$ satisfies (\ref{eq:attwavet1}), then by the completeness of $\psi_k$ in $L^2(\Omega, c^{-2})$ $L^2(\Omega, c^{-2})$-convergent series $u$ satisfies the PDE in (\ref{eq:attwave}).
\end{proof}
To prove Theorem, we need the following proposition:
\begin{prop}\label{prop:decay}
Under assumptions \eqref{eq:assumptions}, there exists a unique global solution $u$
of \eqref{eq:attwave} such that
\begin{equation}\label{eq:decay}
\lim_{t\to \infty}u(\xx,t)=\lim_{t\to \infty}\partial_t u(\xx,t)=0\qquad \mbox{for a.e. }\quad \xx\in \Omega.
\end{equation}
\end{prop}
\begin{proof}
Let us start with energy function:
\[
E_{\operatorname{total},u}(t) = \frac{1}{2} \intL_\Omega \left(  c(\xx)^{-2}|\partial_t u(\xx,t)|^2 + |\nabla u(\xx,t)|^2 \right){\rm d}\xx.
\]
Multiplying the PDE in (\ref{eq:attwave}) by \( \partial_t u \) and integrating over \( \Omega \), we obtain:
\begin{align*}
&\intL_\Omega c(\xx)^{-2} \partial_t^2 u(\xx,t) \partial_t u(\xx,t) {\rm d}\xx + \intL_\Omega \frac{\gamma(t)}{c(\xx)^2} |\partial_t u(\xx,t)|^2 {\rm d}\xx\\
&\qquad  - \intL_\Omega  \Delta u(\xx,t) \partial_t u(\xx,t) {\rm d}\xx 
= - \intL_\Omega c(\xx)^{-2} G_\gamma(\xx,t) \partial_t u(\xx,t) {\rm d} \xx.
\end{align*}
With Green's theorem 
\begin{equation}\label{eq:green1}
\intL_{\Omega}\nabla u(\xx)\cdot \nabla v(\xx){\rm d}\xx=-\intL_{ \Omega}u(\xx)\Delta v(\xx){\rm d}\xx+\intL_{\partial \Omega}u(\yy)\partial_\nu v(\yy){\rm d}s(\yy),
\end{equation}
 this gives:
\[
\frac{\rm d}{{\rm d}t} E_{\operatorname{total},u}(t) + \intL_\Omega\frac{\gamma(t)}{c(\xx)^2} |\partial_t u(\xx,t)|^2{\rm d}\xx = - \intL_\Omega c(\xx)^{-2} G_\gamma(\xx,t) \partial_t u(\xx,t) {\rm d}\xx.
\]
and thus with Young's inequality,  we have
\begin{align*}
\frac{\rm d}{{\rm d}t} E_{\operatorname{total},u}(t) + \intL_\Omega \frac{\gamma(t)}{c(\xx)^2}|\partial_t u(\xx,t)|^2{\rm d}\xx \le\left| \intL_\Omega  c(\xx)^{-2} G_\gamma(\xx,t) \partial_t u(\xx,t) {\rm d}\xx\right|\\
\le \frac{1}{2} \intL_\Omega\frac{\gamma(t)}{c(\xx)^2}|\partial_t u(\xx,t)|^2 {\rm d}\xx + \frac{1}{2} \intL_\Omega \frac{1}{\gamma(t)c(\xx)^2} |G_\gamma(\xx,t)|^2 {\rm d}\xx.
\end{align*}
It can be written as
\begin{align}\label{eq:Ediff}
\frac{\rm d}{{\rm d}t} E_{\operatorname{total},u}(t) \le \frac{\rm d}{{\rm d}t} E_{\operatorname{total},u}(t) + \frac{1}{2} \intL_\Omega \frac{\gamma(t)}{c(\xx)^2} |\partial_t u(\xx,t)|^2 {\rm d}\xx
 \le \frac{1}{2} \intL_\Omega \frac{1}{\gamma(t)c(\xx)^2} |G_\gamma(\xx,t)|^2 {\rm d}\xx.
\end{align}
For any $t\ge0$, integrating from \( 0 \) to \( t \) gives
\begin{equation}\label{eq:E}
E_{\operatorname{total},u}(t) \le E_{\operatorname{total},u}(0) + \frac{1}{2} \intL_0^t \intL_\Omega \frac{1}{\gamma(\tau)c(\xx)^2} |G_\gamma(\xx,\tau)|^2 {\rm d}\xx {\rm d}\tau.
\end{equation}
\begin{lem}\label{lem:gl2}
For the solution $p$ of (\ref{eq:attwaveorigin}) and $g=p|_{\partial\Omega\times[0,\infty)}$  with the assumptions (\ref{eq:assumptions}), we have
$$
E\partial_t^2g\in L^2( \Omega\times[0,\infty))\quad\mbox{and}\quad E\partial_tg\in L^2( \Omega\times[0,\infty)).
$$
\end{lem}
By Lemma \ref{lem:gl2}, we have
\begin{equation}\label{eq:Ggamma}
\intL\half\intL_{\Omega} \frac1{\gamma(t)c(\xx)^2}|G_\gamma(\xx,t)|^2{\rm d}\xx{\rm d}t<\infty,
\end{equation}
and thus the integral on the right-hand side of (\ref{eq:E}) is finite. 
Hence \( E_{\operatorname{total},u}(t)\) remains bounded and thus the solution $u$ is global.
Again integrating (\ref{eq:Ediff}) from \( 0 \) to \(T\) gives
\begin{align*}
&E_{\operatorname{total},u}(T) + \frac{1}{2} \intL_0^T\intL_\Omega \frac{\gamma(t)}{c(\xx)^2} |\partial_t u(\xx,t)|^2 {\rm d}\xx{\rm d}t\le E_{\operatorname{total},u}(0) + \frac{1}{2} \intL_0^T \intL_\Omega \frac{1}{\gamma(t)c(\xx)^2} |G_\gamma(\xx,\tau)|^2 {\rm d}\xx {\rm d}\tau.
\end{align*}
and thus
$$
 \intL_0^\infty\intL_\Omega\frac{\gamma(t)}{c(\xx)^2} |\partial_t u(\xx,t)|^2 {\rm d}\xx{\rm d}t<\infty
$$
and thus for  a.e. fixed $\xx\in \Omega$, $\partial_t u(\xx,t)\to 0$ as $t\to \infty$.

Now we show that for a.e. $\xx\in\Omega$, $u(\xx,t)\to 0$ as $t\to \infty$. Since $\lim_{t\to \infty}\partial_t u(\xx,t)=0$, there is a function $u_\infty$ on $\RR^n$ such that $u(\xx,t)\to u_\infty(\xx)$. Then from (\ref{eq:attwave}) and $\lim_{t\to\infty}G_\gamma (\xx,t)=0$ (by (\ref{eq:Ggamma})), we have
$$
c(\xx)\Delta_{ }u_\infty(\xx)=0 \quad\mbox{and}\quad u_\infty|_{\partial\Omega}=0.
$$
Therefore we have $u_\infty=0$ by strong maximum principle in \cite[Chapter 6]{evans10} and thus $$\lim_{t\to \infty} u(\xx,t)=0.$$

The uniqueness of $u$ follows from (\ref{eq:E}) and the uniqueness of the solution $p$.
\end{proof}
It remains to prove Lemma~\ref{lem:gl2}:
\begin{proof}[Proof of Lemma \ref{lem:gl2}]
Now consider the following energy function: for $D_{\xx}^\alpha=\partial_{x_1}^{\alpha_1} \partial_{x_2}^{\alpha_2}\cdots \partial_{x_n}^{\alpha_n}$, $\alpha=(\alpha_1,\alpha_2,\cdots,\alpha_n)\in \mathbb N\cup \{0\}$ and $|\alpha|=\sum^n_{i=1}\alpha_i\le1$,
\[
E_{\operatorname{total},p}^\alpha (t) = \frac{1}{2} \intL_{\mathbb{R}^n} \left(c(\xx)^{-2} |\partial_t D^\alpha_\xx p(\xx,t)|^2 +  |\nabla D^\alpha_\xx p(\xx,t)|^2 \right){\rm d}\xx+\intL^t_0 \intL_{\mathbb R^n} F_\alpha(\xx,\tau) D^\alpha_\xx \partial_\tau p(\xx,\tau) {\rm d}\xx{\rm d}\tau,
\]
where $F_\alpha(\xx,t)=[D_{\xx}^\alpha, c(\xx)^{-2}](\partial_t^2p(\xx,t)+\gamma(t)\partial_t p(\xx,t))$ and $[A,B]=AB-BA$ is commuter.
Differentiating and using the PDE in (\ref{eq:attwaveorigin}) with Green's Theorem  (\ref{eq:green1}) and compactness of $p$ (\cite[Chapter 7.4]{alinhac09}) yields
\begin{equation}\label{eq:diffE}
\frac{\rm d}{{\rm d}t} E_{\operatorname{total},p}^\alpha (t) = -\intL_{\mathbb{R}^n}\frac{\gamma(t)}{c(\xx)^2}|\partial_t   D^\alpha_\xx p(\xx,t)|^2{\rm d}\xx\le 0.
\end{equation}
Integrating (\ref{eq:diffE})  in time gives
\[
\intL_0^\infty \intL_{\mathbb{R}^n} \frac{\gamma(t)}{c(\xx)^2}|\partial_t D^\alpha_\xx p(\xx,t)|^2{\rm d}\xx{\rm d}t \le 2E_{\operatorname{total},p}^\alpha(0) < \infty,
\]
which implies \( \partial_t p \in L^2([0,\infty);H^1(\RR^n))\) with assumption (\ref{eq:assumptions}).
By Trace Theorem \cite[Chapter 5]{evans10} and the boundedness of the harmonic extension operator $E:H^s(\partial \Omega)\to H^{s+\frac12}(\Omega)$, we have $ \partial_t g\in L^2(\partial\Omega\times[0,\infty))$ and
$$
||E\partial_t g(\cdot,t)||_{L^{2}(\Omega)}\le  ||E\partial_t g(\cdot,t)||_{H^{\frac12}(\Omega)}\le C ||\partial_t  g(\cdot,t)||_{H^{0}(\partial \Omega)},
$$
which implies
$$
E \partial_tg\in L^2(\Omega\times[0,\infty)).
$$

Similarly, considering the following energy functional 
\[
\begin{array}{ll}
E_{\operatorname{high},p}^\alpha (t) =& \displaystyle \frac{1}{2} \intL_{\mathbb{R}^n} \left(c(\xx)^{-2} |\partial_t^2   D^\alpha_\xx p(\xx,t)|^2 +  |\nabla \partial_t   D^\alpha_\xx p(\xx,t)|^2 \right){\rm d}\xx \\
&\qquad \displaystyle+ \intL_{0}^{t}\intL_{\mathbb R^n}
\Big( \frac{\gamma'(\tau)}{c(\xx)^{2}}
      \,\partial_\tau D^\alpha_{\xx} p(\xx,\tau)
      + \partial_\tau F_\alpha(\xx,\tau) \Big)
      \partial_\tau^{2} D^\alpha_{\xx} p(\xx,\tau)\,
      {\rm d}\xx\,{\rm d}\tau,
\end{array}
\]
we have 
\[
\frac{\rm d}{{\rm d}t} E_{\operatorname{high},p} (t) = -\intL_{\mathbb{R}^n}\frac{\gamma(t)}{c(\xx)^2}|\partial_t^2   D^\alpha_\xx p(\xx,t)|^2{\rm d}\xx \le 0.
\]
Notice that from (\ref{eq:attwaveorigin}) with assumptions (\ref{eq:assumptions}), $\partial^2_tD^\alpha_\xx p(\xx,0)\in L^2(\RR^n)$ and $\partial_t D^\alpha_\xx p(\xx,0)=0$ and thus 
\[
\intL_0^\infty \intL_{\mathbb{R}^n} \frac{\gamma(t)}{c(\xx)^2}|\partial_t^2D_{\xx}^\alpha  p(\xx,t)|^2{\rm d}\xx {\rm d}t \le 2 E_{\operatorname{high},p}^\alpha (0) < \infty.
\]
Hence we have $\partial_t^2 p\in L^2([0,\infty);H^1(\RR^n))$.
Again by Trace Theorem \cite[Chapter 5]{evans10}, and the boundedness of the harmonic extension operator, we have
$$
E \partial_t^2g\in L^2(\Omega\times[0,\infty)).
$$

\end{proof}

Now we are ready to prove Theorem \ref{thm:unique}:
\begin{proof}[Proof of Theorem \ref{thm:unique}]
Let us consider the solutions $\tilde u_k$ of  (\ref{eq:attwavet1}) with 
\begin{equation}\label{eq:decay2}
\lim_{t\to \infty}\tilde u_k(t)=\lim_{t\to \infty} \tilde u_k'(t)=0  
\end{equation}
The solution $\tilde u_k$ are determined by only $G_{\gamma,k}$ (independent on $p_0$).

On the other hand, since the Fourier coefficient $u_k$ of the solution $u$ of (\ref{eq:attwave}) with initial and boundary values (\ref{eq:attwavet1})  satisfies $u_k,u'_k\to 0$ as $t\to \infty$, we have $u_k=\tilde u_k$ by uniqueness of ODE solution. 
Hence $p_0$ can be uniquely determined from $\tilde u_k$ and thus $G_{\gamma,k}$.

By strong maximum principle in \cite{evans10}, $Eg$ is unique and thus we have our assertion.
\end{proof}

\subsection{Constant damping}
In the special case where $\gamma$ is constant, we can derive an explicit inversion procedure.
\begin{thm}
Let $B_{k,\pm} = (-\gamma \pm A_k)/2$ and $A_k = (\gamma^2 - 4\lambda_k^2)^{1/2}$.
Then the series $u(\xx,t) = \sum_{k\in\mathbb{N}} u_k(t)\psi_k(\xx)$ converges in $L^2(\Omega,c^{-2})$ and
its sum is a solution of \eqref{eq:attwave} with the assumptions (\ref{eq:assumptions}), where
\[
u_k(t) = A_k^{-1}\intL_t^\infty \big(e^{B_{k,+}(t-\tau)} - e^{B_{k,-}(t-\tau)}\big) G_{\gamma,k}(\tau)\,{\rm d}\tau.
\]
Moreover, the initial pressure $p_0$ can be reconstructed as
\begin{align*}
&p_0(\xx)=Eg(\xx,0)+\sum_{k\in\mathbb N}A_k ^{-1}\intL^\infty_0  (e^{-B_{k,+} \tau}-e^{-B_{k,-} \tau})G_{\gamma,k}(\tau){\rm d}\tau \psi_k(\xx).
\end{align*}
\end{thm}
\begin{proof}

To check that $u_k$ is a solution of  (\ref{eq:attwavet1}), we consider $u_k' $ and $u_k '' $:
$$
\begin{array}{ll}
u_k'(t)\displaystyle=A_k ^{-1}\intL^\infty_t  (B_{k,+} e^{B_{k,+} (t-\tau)}-B_{k,-} e^{B_{k,-} (t-\tau)})G_{\gamma,k}(\tau){\rm d}\tau
\end{array}
$$
and
$$
\begin{array}{ll}
u_k''(t)\displaystyle=-G_{\gamma,k}(t)\displaystyle+A_k ^{-1}\intL^\infty_t  (B_{k,+} ^2e^{B_{k,+} (t-\tau)}-B_{k,-} ^2e^{B_{k,-} (t-\tau)})G_{\gamma,k}(\tau){\rm d}\tau.
\end{array}
$$
Then using the fact that $(B_{k,\pm} )^2+\gamma B_{k,\pm} +\lambda_k^2=0$, we have $u_k$ satisfies (\ref{eq:attwavet1}).

All in all, $u(\xx,t)=\sum u_k(t)\psi_k(\xx)$ satisfies the PDE in (\ref{eq:attwave}) with
$$
\lim_{t\to \infty}u_k(t)=\lim_{t\to \infty}\partial_t u_k(t)=0\qquad \mbox{for any }\xx\in \Omega, k\in \mathbb N.
$$
By uniqueness of the solution, we have that $u$ is the solution of (\ref{eq:attwave}).
\end{proof}
We now eliminate the explicit appearance of the extension operator $E$:
\begin{align*}
G_{\gamma,k}(t)&=\displaystyle \intL_{\Omega}G_\gamma(\xx,t) \overline{\psi_k(\xx)}c(\xx)^{-2}{\rm d}\xx=
\displaystyle \lambda^{-2}_k\intL_{\Omega}G_\gamma(\xx,t)\Delta_{\xx} \overline{\psi_k(\xx)}{\rm d}\xx\\
&=\displaystyle \lambda^{-2}_k\intL_{\partial \Omega}(\partial_{t}^2 g(\yy,t)+\gamma(\yy)\partial_{t} g(\yy,t)) \partial_\nu \overline{\psi_k(\yy)}{\rm d}s(\yy),
\end{align*}
where we used Green's theorem, harmonicity of $Eg$ and $\psi_k=0$ on $\partial \Omega$, again. 
Similarly, we have 
$$
Eg(\xx,0)=\sum_k\lambda^{-2}_k\intL_{\partial \Omega}g(\yy,0) \partial_\nu \overline{\psi_k(\yy)}{\rm d}s(\yy) \psi_k(\xx),
$$
in  $L^2(\Omega, c^{-2})$ sense.
\begin{cor}
Under assumptions \eqref{eq:assumptions}, the initial function $p_0$  inside $\Omega$  can be reconstructed from the data $g$ in (\ref{eq:attwaveorigin}) as the following $L^2(\Omega, c^{-2})$-convergent series:
$$
p_0(\xx)=\sum_k\lambda^{-2}_k p_{0,k}\psi_k(\xx),
$$
where
\begin{align*}
&p_{0,k}=\intL_{\partial \Omega}g(\yy,0) \partial_\nu \overline{\psi_k(\yy)}{\rm d}\yy\\
&\qquad\qquad+
A_k ^{-1}\intL^\infty_0  \intL_{\partial \Omega}(e^{-B_{k,+} t}-e^{-B_{k,-} t})
(\partial_{t}^2 g(\yy,t)+\gamma(\yy)\partial_{t} g(\yy,t)) \partial_\nu \overline{\psi_k(\yy)}{\rm d}s(\yy){\rm d}t,
\end{align*}
where ${\rm d}s(\yy)$ denotes the surface measure on $\partial \Omega$, again.
\end{cor}
\section{Numerical inversion procedure}\label{sec:control}
To solve for $p_0(\xx)$ given observation function $g(\xx,t)$ on the observation domain $\partial \Omega$, we consider the following optimization problem:
\begin{equation}\label{eq:minproblem}
\begin{aligned}
\min_{p_0\in P_{ad}}& J(p_0,p):= \dfrac{1}{2}\intL_0^T\intL_{\partial\Omega} (p(\xx,t) - g(\xx,t))^2~{\rm d}s(\xx){\rm d}t + \dfrac{\alpha}{2} \intL_\Omega |p_0(\xx)|^2~{\rm d}\xx + \beta \intL_\Omega |p_0(\xx)|~{\rm d}\xx\\
\mbox{such that } &\partial_t^2 p(\xx,t) + \gamma(t)\partial_t p(\xx,t) - c(\xx)^2 \Delta p(\xx,t) = 0, \quad (\xx,t) \in \mathbb{R}^n \times (0, T),\\
&p(\xx,0) = p_0(\xx),\quad \xx \in \mathbb{R}^n,\\
&\partial_tp(\xx,0) = 0,\quad \xx \in \mathbb{R}^n,
\end{aligned}
\end{equation}
where $p_0\in P_{ad} = \lbrace L^2(\Omega): 0 \leq p_l \leq p_0(\xx) \leq p_r \rbrace.$ Here the first term represents the standard least-squares data fitting term, and the last two terms represent the $L^2-L^1$ regularization term, with $\alpha,\beta>0$, that has the ability to reconstruct sparsity patterns in $p_0(\xx)$.

We first remark that since the functional $J$ is convex with respect to $p_0$ since the attenuated wave equation is a linear function of $p_0$. Thus, there is exists a solution pair $
(p_0^*,p^*)$ to the minimization problem \eqref{eq:minproblem}. To characterize the solutions of \eqref{eq:minproblem}, we use the framework of the Pontryagin's maximum principle (PMP). For this purpose, we formulate the following Hamiltonian function:
\begin{equation}\label{eq:Hamiltonian}
H(\xx,p_0(\xx),q(\xx,:)) =  \dfrac{\alpha}{2} |p_0(\xx)|^2 + \beta  |p_0(\xx)| + p_0(\xx) [ \partial_tq(\xx,0)-\gamma(0)q(\xx,0)],\quad \xx \in \Omega,
\end{equation}
where $q(\xx,t)$ solves the following adjoint equation
\begin{equation}\label{eq:adjoint}
\begin{aligned}
&\partial_t^2 q(\xx,t) - \partial_t [\gamma(t)q(\xx,t)] -  \Delta [c(\xx)^2q(\xx,t)] = -[p(\xx,t)-g(\xx,t)]\chi_{\partial\Omega}, \quad (\xx,t) \in \mathbb{R}^n \times (0, T)\\
&q(\xx,T) = 0,\quad \xx \in \mathbb{R}^n ,\\
&\partial_tq(\xx,T) = 0,\quad \xx \in \mathbb{R}^n .
\end{aligned}
\end{equation}
We remark that the spatial domain for $\xx$ in the definition of $H$ is chosen to be $\Omega$, corresponding to the region of interest. Then, we have the following characterization of the optimal control through the PMP:
\begin{thm}\label{th:PMP}
The optimal initial condition and adjoint $(p_0^*,q^*)$ satisfies the following PMP criterion 
\[
H(\xx, p_0^*(\xx),q^*(\xx,:)) = \min_{v\in[p_l,p_r]} H(\xx, v,q^*(\xx,:)),~ \forall \xx \in \Omega.
\]
\end{thm}
For the proof of this thoerem, we use the standard needle variation technique (see, e.g.,  \cite{Borzi2023}). 
Let $S_{\varrho}\left(\xx_{0}\right)$ be an open ball centered at $\xx_{0} \in \Omega$ with radius $\varrho$ 
such that the Lebesgue measure of the ball tends to zero as 
$\varrho \rightarrow 0 $,  
$\lim_{\varrho\rightarrow 0}|S_{\varrho}\left(\xx_{0}\right)|=0$. We define the needle variation at $\xx_0$ of  $p_0  \in P_{ad}$ as follows:  
\begin{equation}
p_0^{\varrho}\left(\xx\right):= \begin{cases}
p_0(\xx) & \mathrm{\ on\ } \, \Omega \backslash S_{\varrho}\left(\xx_{0}\right)\\
w & \mathrm{\ in\ } \, S_{\varrho}\left(\xx_{0}\right)\cap \Omega
\end{cases}\label{eq:Def of u_k}
\end{equation}
where $w \in [p_l,p_r]$. We note that $p_0^{\varrho}\in P_{ad}$ for all $\xx_0 \in \Omega$, and for any given $p_0 \in P_{ad}$ (see \cite{Borzi2023}).

Since $p_0\left(\xx\right)\in L^{\infty}\left(\Omega\right)$, almost every where in $\Omega$ is a Lebesgue point of 
$p_0\left(\xx\right)$. Therefore, it holds: 
\begin{equation}
\|p_0^{\varrho}-p_0\|_{L^{2}(\Omega)} =\left(\intL_{S_{\varrho}\left(\xx_{0}\right)}|w-p_0\left(\xx\right)|^{2} \,{\rm d} \xx\right)^{\frac{1}{2}}{\rightarrow}0,~ \mbox{as $\varrho\rightarrow 0$},
\label{limurho}
\end{equation}
for almost all $\xx_0 \in \Omega$. From this fact and the 
stability result for the attenuated wave equation, we have: 
\begin{equation}
\|p_{\varrho}-p\|_{L^\infty(0,T; L^2(\Omega))} \rightarrow 0, ~ \|q_{\varrho}-q\|_{L^\infty(0,T; L^2(\Omega))} \rightarrow 0,~\|\partial_tq_{\varrho}-\partial_tq\|_{L^\infty(0,T; L^2(\Omega))} \rightarrow 0,
\label{limpqrho}
\end{equation}
and 
\begin{equation}
 \| \partial_t p_{\varrho} (\cdot,0) - \partial_t p (\cdot,0) \|_{L^2(\Omega)} 
  \rightarrow 0,
\qquad 
 \| p_{\varrho}(\cdot,0) -p (\cdot,0) \|_{L^2(\Omega)} 
  \rightarrow 0,
\label{limpqrhoBc}
\end{equation}
where $p_\varrho$ and $q_\varrho$ denote the state and adjoint variables corresponding to $p_0^\varrho$. 

Let $p_0$ be a minimizer. We then compute the difference of cost functionals to obtain 
\begin{align}
J(p_0^\varrho,p_\varrho) - J(p_0,p)  
& =  \intL_0^T \intL_{\partial\Omega} \Big( \frac{p_\varrho(\yy,t) 
+ p(\yy,t)}{2} - g(\yy,t) \Big) \, \big( p_\varrho(\yy,t) - p(\yy,t) \big) \,{\rm d} s(\yy) \, {\rm d}t \notag \\
& +  \intL_{\Omega} \big( G(p_0^\varrho(\xx)) - G(p_0(\xx))   \big) \, {\rm d}\xx,
\label{eDiffJ}
\end{align} 
where $G(p_0) = \dfrac{\alpha}{2} \intL_\Omega |p_0(\xx)|^2~{\rm d}\xx + \beta \intL_\Omega |p_0(\xx)|~{\rm d}\xx.$
We now introduce an intermediate adjoint problem as follows: 
\begin{equation}\label{eWaveAdjxInter}
\begin{aligned}
\partial_{t}^2 q_\varrho(\xx,t) - \partial_t [\gamma(t)q_\varrho(\xx,t)] -  \Delta [c(\xx)^2q_\varrho(\xx,t)] &= -\left[\frac{p_\varrho(\xx,t) 
+ p(\xx,t)}{2}-g(\xx,t)\right]\chi_{\partial\Omega}, \\
&\qquad\qquad\qquad \qquad\quad (\xx,t) \in \mathbb{R}^n \times (0, T),\\
q_\varrho(\xx,T) &= 0,\quad \xx \in \mathbb{R}^n ,\\
\partial_tq_\varrho(\xx,T) &= 0,\quad \xx \in \mathbb{R}^n .
\end{aligned}
\end{equation}
Standard regularity estimates of \eqref{eWaveAdjxInter} gives us the following results: 
\begin{equation}
q_\varrho \in  L^\infty(0,T; L^2(\Omega)), \qquad \partial_t q_\varrho \in L^\infty(0,T; L^2(\Omega)) \cap L^2(0,T; L^2(\Omega)).
\label{eStimate}
\end{equation}
We can state the following lemma related to the difference in the value of the functional $J$ at the needle variation:
\begin{lem}\label{eDifJDifH}
The following equation holds 
\[
J(p_0^\varrho,p_\varrho) - J(p_0,p)  = - \intL_{\Omega} \left( H\left(\xx,p_0^{\varrho},q_\varrho\right)-H\left(\xx,p_0,q_\varrho\right) \right) \, {\rm d}\xx, 
\]
where $H$ is given as in \eqref{eq:Hamiltonian}.
\end{lem}
\begin{proof}
Using \eqref{eWaveAdjxInter} in \eqref{eDiffJ}, to replace 
$\Big( \frac{p_\varrho(\xx,t) 
+ p(\xx,t)}{2} - g(\xx,t) \Big)$  with the corresponding 
left-hand side of the intermediate adjoint problem, 
integration by parts and the use of the boundary and initial and terminal conditions give the desired result. 
\end{proof}

Now, we can consider the needle variation in the limit $\varrho \to 0$. We have
\begin{lem}\label{eMon}
Let $p_0^*\in P_{ad}$ be a minimizer and $w\in [p_l,p_r]$. 
Furthermore, let $p_0^{\varrho}$ be defined as in
\eqref{eq:Def of u_k}, and $p_{\varrho}$ be the
solution to the attentuated wave equation \eqref{eq:attwave} with $p_0= p_0^{\varrho}$. Then, the following holds
$$
0 \le  \lim_{\varrho\rightarrow 0}\frac{1}{|S_{\varrho}\left(\xx_0\right)|}   \left(J\left({p_0}^{\varrho},p_{\varrho}\right)-J\left(p_0^*, p^*\right)\right)=  H\left(\xx_0,p_0^*,q^*\right)  -H\left(\xx_0,w,q^*\right) , 
   $$
for almost all $\xx_0 \in \Omega$ and $w \in [p_l,p_r]$. 
\end{lem}
\begin{proof}
For any $k\in\mathbb{N}$, we have
\[
\begin{aligned}
0 \leq &\frac{1}{|S_{\varrho}\left(\xx_0\right)|}   \left(J\left(p_0^{\varrho},p_{\varrho}\right)-J\left(p_0^*,p^*\right)\right)\\
=& \frac{1}{|S_{\varrho}\left(\xx_0\right)|}  \intL_\Omega \Big( H\left(\xx,p_0^*,q_\varrho\right)-H\left(\xx,p_0^{\varrho},q_\varrho\right) \, {\rm d}\xx\\
= & \frac{1}{|S_{\varrho}\left(\xx_0\right)|}  \intL_{S_{\varrho}\left(\xx_0\right)}
\Big( H\left(\xx,p_0^*,q^*\right)  -H\left(\xx,w,q^*\right) \Big) \,{\rm d} \xx\\
&+\frac{1}{|S_{\varrho}\left(\xx_0\right)|}  \intL_{S_{\varrho} \left(\xx_0\right)} \Big( (p_0^*-w)[(\partial_tq_{\varrho}(\xx,0)-\partial_tq^*(\xx,0)) -\gamma(0)(q_\varrho(\xx,0)-q^*(\xx,0))] \Big) \, {\rm d}\xx.
\end{aligned}
\]
Since, $\xx \in \Omega \mapsto \Big( H\left(\xx,p_0^*,q^*\right)  -H\left(\xx,w,q^*\right) \Big) \in L^1(\Omega)$, convergence results for $p_0^\varrho$, $p_\varrho$, and $q_\varrho$ in \eqref{limurho} and \eqref{limpqrho}, we have the following using the mean value theorem: 
\[
0 \le  \lim_{\varrho\rightarrow 0}\frac{1}{|S_{\varrho}\left(\xx_0\right)|}   \left(J\left(p_0^{\varrho},p_{\varrho}\right)-J\left(p_0^*, p^*\right)\right)
=  H\left(\xx_0,p_0^*,q^*\right)  -H\left(\xx_0,w,q^*\right).
\]
\end{proof}
As a consequence of Lemma \ref{eMon}, we have proved the PMP Theorem \ref{th:PMP}.

A major advantage of the PMP framework is that the PMP characterization does not involve derivatives of the objective functional $J$ with respect to $p_0$, unlike the standard first order optimality conditions using the Euler-Lagrange framework. This motivates the construction of the sequential quadratic Hamiltonian method (SQH) to implement the PMP condition as stated in Theorem \ref{th:PMP}. The SQH method, developed in its recent form in \cite{Borzi2023} represents a recent development in the field of
successive approximations (SA) schemes. The working principle of these methods is the iterative pointwise minimisation of
the Hamiltonian function of the given minimization problem. The starting point of the SQH method is the augmented Hamiltonian function, which is defined as:
\[
H_\epsilon(\xx,p_0,\tilde{p}_0, q) = H(\xx,p_0,q) + \epsilon (p_0 - \tilde{p}_0)^2,
\label{eq:augmented_Hamiltonian}
\]
where $\epsilon > 0$ is a penalization parameter that is adaptively adjusted at each iteration of the SQH process. Specifically, $\epsilon$ is increased if a sufficient decrease in the functional $J$ is not observed, and decreased if $J$ decreases adequately. Here, $\tilde{p}_0$ represents the previous approximations of the initial condition $p_0$, respectively. The purpose of the quadratic term $\epsilon (p_0 - \tilde{p}_0)^2 $ is to ensure that the pointwise minimizer of $H_\epsilon$, and thus the updates to $p_0$ remain close to the prior values $\tilde{p}_0$, especially when $\epsilon$ is large. It is important to note that during each optimization step, specifically during a minimization sweep over all spatial grid points $\xx$, the values of $p$ and $q$ used are those obtained from the previous iteration. The SQH algorithm is outlined in the following algorithm:

\begin{algorithms}[SQH method]\text{ }
\begin{itemize}
\item Input:  initial approx. $p_0^0$, max. number of iterations $k_{max}$, tolerance $\kappa >0$, 
$\epsilon>0$, $\lambda >1$, $\eta > 0$, and $\zeta\in\left(0,1\right)$; 
set $\tau > \kappa $, $k:=0$.
\item Compute the solution $p^0$ to the damped wave equation given in \eqref{eq:minproblem} with initial condition $p_0=p_0^0$. 
\item While ($k<k_{max} ~\&\& ~ \tau > \kappa$ ) do 
\begin{enumerate}[label=(\alph*)]
\item Compute the solution $q^k$ to the adjoint problem \eqref{eq:adjoint} with $p=p^k$. 
\item Determine $p_0^{k+1}$ such that the following optimization problem is satisfied 
$$
{H}_\epsilon \left(\xx,p_0^{k+1},p_0^k, q^k\right)
= \min_{v\in [p_l,p_r]} {H}_\epsilon \left(\xx,v,p_0^k, q^k\right), 
$$
at almost all $\xx \in \Omega$.
\item Compute the solution $p^{k+1}$ to the damped wave equation given in \eqref{eq:minproblem} with initial condition $p_0=p_0^{k+1}$
\item Compute $\tau:=\|p_0^{k+1 } -p_0^{k}\|^2_{L^{2}\left(\Omega\right)}$.
\item If $J\left(p_0^{k+1},p^{k+1}\right)-J\left(p_0^{k},p^{k}\right)> -\eta \, \tau$, 
then increase $\epsilon$ with $\epsilon=\lambda \, \epsilon$ and go to Step (b).\\
Else if 
$J\left(p_0^{k+1},p^{k+1}\right)-J\left(p_0^{k},p^{k}\right) \leq -
\eta \, \tau $, then 
decrease $\epsilon$ with $\epsilon=\zeta \, \epsilon$ and continue. 
\item Set $k:=k+1$. 
\end{enumerate}
\item end While
\end{itemize}
\label{algSQHmethod}
\end{algorithms}

In Step {\it (e)} of this algorithm, if the inequality $J\left(p_0^{k+1},p^{k+1}\right)-J\left(p_0^{k},p^{k}\right)> -\eta \, \tau$ holds, it indicates that a sufficient decrease in the objective functional $J$ has not been achieved. In such a case, $\epsilon$ is increased (since $\lambda > 1$), and the optimization in Step {\it (b)} is repeated with the updated augmented Hamiltonian function. In contrast, if the inequality does not hold, it confirms that the required reduction in $J$ has been achieved. The updated initial condition $p_0^{k+1}$ is then adopted, together with the corresponding updates $p^{k+1}$ and $q^{k+1}$ for the damped wave equation and its adjoint. In this situation, $\epsilon$ is reduced by a factor $\zeta < 1$.

\begin{thm} \label{th:SQH} Let $\left(p_0^{k},p^{k}\right)$ and $\left(p_0^{k+1},p^{k+1}\right)$ be generated by the SQH method (Algorithm \ref{algSQHmethod}) applied to \eqref{eq:minproblem}, with $p_0^{k+1}, p^k \in P_{ad}$. Then, there exists a constant $C > 0$, independent of $\epsilon$ and $p_0^k$, such that for the current value of $\epsilon > 0$ chosen by Algorithm \ref{algSQHmethod}, the following inequality holds:
\begin{equation}\label{eq:ineq}
J\left(p_0^{k+1},p^{k+1}\right) - J\left(p_0^{k},p^{k}\right) \leq -\left(\epsilon - C\right) \, \| p_0^{k+1} - p_0^k \|^{2}_{L^2\left(\Omega\right)}.
\end{equation}
In particular, this implies $J\left(p_0^{k+1},p^{k+1}\right) - J\left(p_0^{k},p^{k}\right) \leq - \eta \, \tau$ for $\epsilon \geq C + \eta$ and $\tau = \| p_0^{k+1} - p_0^k \|^{2}_{L^2\left(\Omega\right)}$. \end{thm}
\begin{proof}
 We have the following Hamiltonian function 
$$
H\left(\xx,p_0(\xx),q(\xx)\right) = \dfrac{\alpha}{2} |p_0(\xx)|^2 + \beta  |p_0(\xx)| + p_0(\xx) [ \partial_tq(\xx,0)-\gamma(0)q(\xx,0)].
$$ 
In Step {\it (b)} of Algorithm \ref{algSQHmethod}, we have that for almost all $\xx \in \Omega$ it holds:
$$
{H}_\epsilon \left(\xx,p_0^{k+1},p_0^k, q^k\right)
\leq {H}_\epsilon \left(\xx,v,p_0^k, q^k\right) , 
$$ 
for all $v\in [p_l,p_r]$. Therefore, we have
\begin{equation*}
{H}_\epsilon \left(\xx,p_0^{k+1},p_0^k, q^k\right)
\leq {H}_\epsilon \left(\xx,p_0^k,p_0^k, q^k\right)={H} \left(\xx,p_0^k, q^k\right) .
\end{equation*}
Hence, we obtain the inequality
\begin{equation}\label{ePartialMax}
{H} \left(\xx,p_0^{k+1}, q^k\right)+\epsilon \, | p_0^{k+1} - p_0^k |^2 \le {H} \left(\xx,p_0^k, q^k\right) .
\end{equation}
Define $\delta p = p^{k+1}-p^k$, $l(p_0) = \dfrac{\alpha}{2} |p_0|^2 + \beta |p_0|$
and  $\delta p_0 = p_0^{k+1}-p_0^k$. We have 
\begin{align*}
&J(p_0^{k+1},p^{k+1}) - J(p_0^{k},p^{k}) \\
&\quad=
 \dfrac{1}{2}\intL_0^T\intL_{\partial\Omega} 2(p^k(\yy,t) - g(\yy,t))\delta p(\yy,t) + (\delta p(\yy,t))^2 ~{\rm d}s(\yy){\rm d}t + \intL_\Omega l(p_0^{k+1})(\xx)-l(p_0^k) (\xx)~{\rm d}\xx \\
&\quad=\dfrac{1}{2}\intL_0^T\intL_{\partial\Omega} (\delta p(\yy,t))^2 ~{\rm d}s(\yy){\rm d}t + \intL_0^T\intL_{\partial\Omega} (p^k(\yy,t) - g(\yy,t))\delta p(\yy,t) ~{\rm d}s(\yy){\rm d}t\\
&\quad\qquad\qquad\qquad\qquad\qquad + \intL_\Omega l(p_0^{k+1})(\xx)-l(p_0^k)(\xx) ~{\rm d}\xx\\
&  \quad=\dfrac{1}{2}\intL_0^T\intL_{\partial\Omega} (\delta p(\yy,t))^2 ~{\rm d}s(\yy){\rm d}t + \intL_{\Omega} [\partial_tq^k(\xx,0)-\gamma(0)q^k(\xx,0)]\delta p_0(\xx) ~{\rm d}\xx \\
&\quad\qquad\qquad\qquad\qquad\qquad+ \intL_\Omega l(p_0^{k+1})(\xx)-l(p_0^k)(\xx) ~{\rm d}\xx\\
&\quad=\dfrac{1}{2}\intL_0^T\intL_{\partial\Omega} (\delta p(\yy,t))^2 ~{\rm d}s(\yy){\rm d}t + \intL_{\Omega} {H} \left(\xx,p_0^{k+1}, q^k\right) - {H} \left(\xx,p_0^{k}, q^k\right) ~{\rm d}\xx \\
&\quad \le 
C \,  \| p_0^{k+1} - p_0^k \|^2_{L^2(\Omega)}  - \epsilon \,  \| p_0^{k+1} - p_0^k \|^2_{L^2(\Omega)} ,
\end{align*}
where the last step follows from the equation following \eqref{ePartialMax} and the standard stability estimate of solutions of \eqref{eq:attwave}. 
\end{proof}

This theorem shows that, if $(p_0^k, p^k)$ are not already optimal, 
it is possible to choose $\epsilon > C$ to obtain a successful minimization step. We have the following corollary as a consequence of Theorem \ref{th:SQH}.

\begin{cor} \label{CorollaryEps}
The sequence $\lbrace\epsilon\rbrace$ of the SQH iterates is bounded.
\end{cor}
\begin{proof}
From Algorithm \ref{algSQHmethod}, we note that the successful 
$k$-th minimization step is performed with $\epsilon= \, (C + \eta)$, 
where $C$ is estimated in Theorem \ref{th:SQH} above. 
Then, there is a constant $\bar{\epsilon} $ related to the stability estimate of the attenuated wave equation \eqref{eq:attwave}, 
the data of the problem, and the regularization and SQH parameters, that 
provides an upper bound $\epsilon \le \bar{\epsilon}$ of the sequence $\lbrace\epsilon\rbrace$ of the SQH iterates.
\end{proof}
This corollary ensures that the SQH algorithm converges, which is what we prove in the next theorem.
\begin{thm}
\label{Tsqh2}
Under the assumptions of Theorem \ref{th:SQH}, if in Algorithm \ref{algSQHmethod}, at every $k$-th iterate, $\epsilon=C+\eta$ 
is chosen (and tolerance  $\kappa=0$, so that the algorithm never stops), then the following holds
\begin{enumerate}
\item $\lim\limits_{k \to \infty} | J(p_0^{k+1},p^{k+1}) - J(p_0^{k},p^k)|=0$
	\label{sqhac}
\item $\lim\limits_{k \to \infty} \| p_0^{k+1} - p_0^{k} \|_{ L^2 \left( \Omega \right) }=0$. 
	\label{sqhbc}
\end{enumerate}
\end{thm}
\begin{proof}
The first statement follows from \eqref{eq:ineq}, since $\epsilon$ can be chosen greater than $C$, which would imply that the sequence $\lbrace J(p_0^{k},p^k)\rbrace $ is monotonically decreasing the $\mathbb{R}$ and is a Cauchy sequence, hence convergent. To prove \ref{sqhbc}, we rewrite 
\eqref{eq:ineq} as follows 
$$
 \| p_0^{k+1} - p_0^{k} \|^{2}_{ L^2 \left( \Omega\right) } \le \frac{1}{\eta}
 \Big[ J\left(p_0^k, p^{k}\right)   - J\left(p_0^{k+1},p^{k+1}\right) \Big] . 
$$
Therefore we have the partial sum
$$
 \sum_{k=0}^K \| p_0^{k+1} - p^{k} \|^{2}_{ L^2 \left( \Omega \right) } \le \frac{1}{\eta}
 \Big[ J\left(p_0^0,p^{0}\right)   - J\left(p_0^{K+1},p^{K+1}\right) \Big].
$$
This shows that in the limit $K \to  \infty$, the series with positive elements 
$\| p_0^{k+1} - p_0^{k} \|^{2}_{ L^2(\Omega)}$ is convergent, which proves the result. \end{proof}

Theorem \ref{Tsqh2} guarantees that Algorithm \ref{algSQHmethod} is well defined for $\kappa>0$. Hence, there is an iteration number $k_0\in\mathbb{N}$ such that $\Vert p_0^{k_0+1}-p^{k_0}\Vert_{L^2(\Omega)}\leq\kappa$. This implies that the 
SQH algorithm stops in finitely many steps.

\subsection{CNN-guided initial guess construction}

To solve for $p_0$ using the SQH scheme, we construct the initial guess, using a CNN algorithm. The architecture of the CNN is given as follows: The CNN begins with an input convolutional layer consisting of 32 filters with a kernel size of 3 and a stride of 2, using the ReLU activation function. A MaxPooling layer with a pooling size of 2 is applied immediately afterward. Max pooling reduces the temporal resolution of the feature maps by retaining only the maximum value within each non-overlapping window of size 2, thereby providing a form of translational invariance and mitigating overfitting by discarding less relevant activations. The network then incorporates three additional convolutional layers, each employing 64 filters, a kernel size of 3, and a stride of 2, with ReLU activation. The first two of these layers are each followed by a MaxPooling layer (with pool size 2), further downsampling the feature representation and emphasizing the most salient local features. After the final convolutional stage, the output tensor is flattened into a one-dimensional vector. This vector is passed through four fully connected dense layers, each containing 64 neurons activated by ReLU, enabling the extraction of increasingly abstract nonlinear feature relationships. The architecture concludes with a dense output layer employing a linear activation function.

The output of the CNN and the output of the time-reversal are summed up to provide the initial guess for the SQH algorithm.

\section{Numerical results}\label{sec:numericalresult}
In this section, we present the results of our SQH method to solve the inverse problem to obtain the initial damped wave acoustic pressure $p_0$ from observational data. We first present the results in a one-dimensional setup. For this purpose, we choose our domain $\Omega = (-1,1)$ and the final time of observations as $T=1$. The complete observation domain $\partial\Omega = \{x=-1,1\rbrace\}$.  We choose a non-trapping sound speed $c(x) = 1+w(x)*0.1*\cos(2\pi x)$, where $w(x)$ is a mollifier centered at the middle of the domain with radius $\sqrt{0.5}$, achieving a maximum value of 1. The damping coefficient $\gamma(t)$ is chosen as $\exp(-t)$. For the spatial grid, we choose 200 points, whereas our time grid comprises of 200 points. To generate the data, we solve the wave equation \eqref{eq:attwave} in free space on a spatial grid with 50 points and on the temporal grid at 50 time points, and then interpolate the solution on the original grid to collect the data on the boundary points $x=-1,1$. The choice of the regularization weights are $\alpha \in [0.1,0.4]$ and $\beta \in [0.001,0.01]$. We then add 10\% additive Gaussian noise to the data.

For training the CNN, we generated a dataset of \(750\) samples: For the first 150 samples, the output was chosen as Gaussian functions with centers randomly drawn from the interval \( (-0.5, 0.1) \cup (0.3, 0.7) \) and widths drawn from \((50, 70) \cup (120, 150)\). For the next 350 samples, the output was chosen as characteristic functions with centers in the interval $(-0.7,-0.1)$ and widths in the interval $(0.1,0.7)$. For the last 250 samples, we have a sum of a Gaussian and characteristic function as output, with centers of the Gaussian drawn from the interval $(-0.9,0.9)$ and widths in $(50, 150)$ while the centers of the characteristic function are in $(-0.9,0.9)$ and widths in $(0.1,0.3)$. The number of epochs chosen for the training was 500, with batch size 32. The optimizer was chosen as ``Adam" with the Huber loss function. The Huber loss offers a balance between the commonly used Mean Square Error (MSE) and Mean Absolute Error (MAE), since it is quadratic for small errors and linear for large errors, making it robust to outliers while still allowing for accurate fits. Unlike MAE, it is smooth and differentiable everywhere, which leads to more stable and efficient optimization. Overall, it provides better gradient behavior and improved performance on noisy data compared to MSE or MAE alone.

We compare our reconstructions with the time-reversal method, obtained by solving the wave equation \eqref{eq:attwave} backwards in time with boundary condition as the observation data, and the reconstructions obtained using our CNN. The time-reversal solution at the final time gives an approximation to the initial condition to the original wave equation \eqref{eq:attwave}.  For qualitative comparison, we also use the quantitative figure of merits: Mean Square Error (MSE), the Peak Signal-to-Noise Ratio (PSNR), and the Structural Similarity Index Measure (SSIM) as defined below:
\[
MSE(I^1,I^2) = \dfrac{1}{n}\sum_{p =1}^n (I_p^1-I^2_p)^2,~ PSNR(I^1,I^2) = 10\log_{10}\left(\dfrac{[\max({I^1,I^2})]^2}{MSE(I^1,I^2)}\right),
\]
where $p$ represents a pixel and $n$ is the total number of pixels, and 
\[
SSIM(I^1, I^2) 
= \frac{(2\mu_{I^1} \mu_{I^2} + C_1)(2\sigma_{I^1I^2} + C_2)}
       {(\mu_{I^1}^2 + \mu_{I^2}^2 + C_1)(\sigma_{I^1}^2 + \sigma_{I^2}^2 + C_2)},
\]
where $\mu_{I^1}$ and $\mu_{I^2}$ denote the average intensities of the 
images $I^1$ and $I^2$, respectively, $\sigma_{I^1}^2$ and $\sigma_{I^2}^2$ represent 
the corresponding intensity variances, and $\sigma_{I^1I^2}$ is the covariance between the two 
images. IMMSE quantifies the average squared pixel-wise difference between two images and is purely error-based, without considering human visual perception. PSNR is a logarithmic transformation of MSE that expresses the fidelity of a reconstructed image relative to the maximum possible pixel intensity; higher PSNR indicates lower error, but it still does not account for structural content. SSIM, on the other hand, compares images based on luminance, contrast, and structural information, making it more aligned with human visual perception. Thus, IMMSE and PSNR focus on pixel-wise accuracy, while SSIM evaluates perceptual similarity.

For the first test case, we choose the true $p_0$ to be a Gaussian  centered at $x=0.5$ with a peak value of 1.0 and standard deviation 0.25. The reconstructions are given in Figure \ref{fig:1DGaussian}.

\begin{figure}[H]
\centering
\subfloat[True]{\includegraphics[width=0.24\textwidth, height=0.24\textwidth]{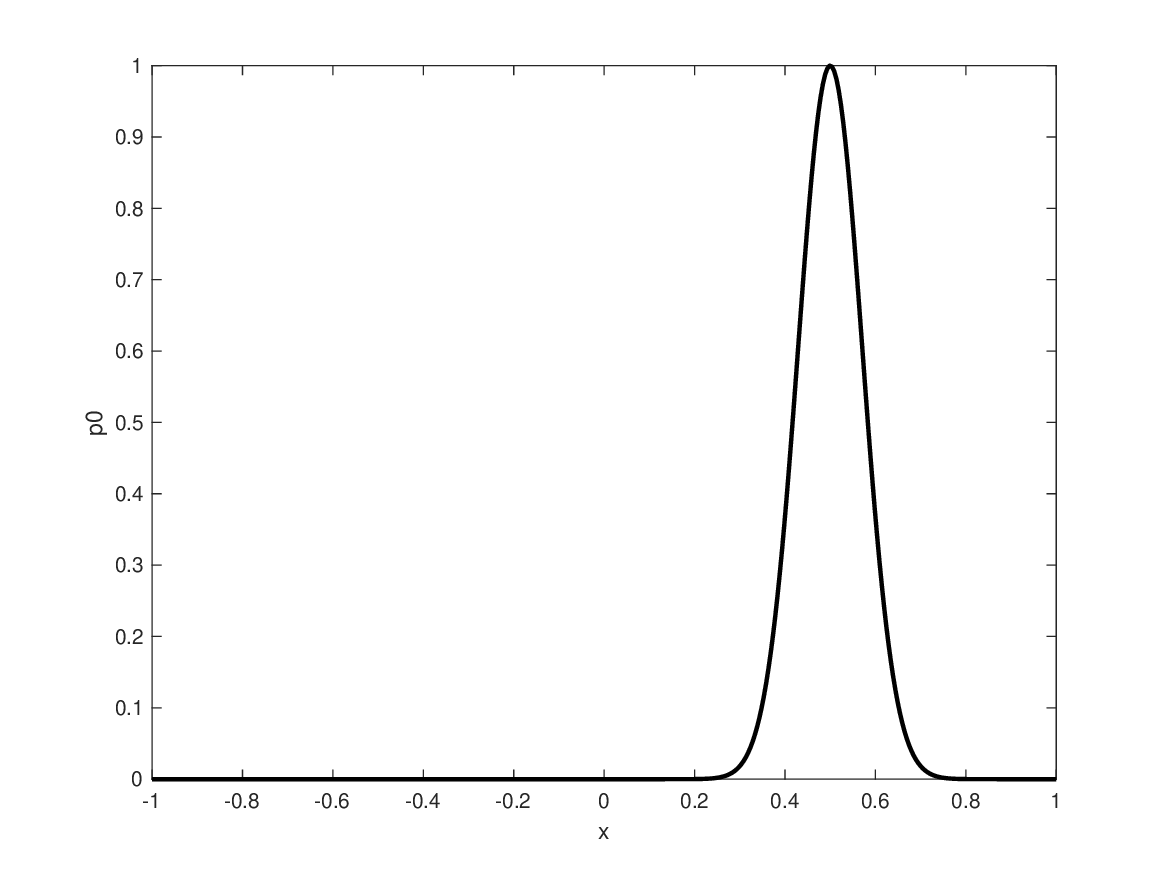} \label{fig:true1DG}}
\subfloat[Time-reversal]{\includegraphics[width=0.24\textwidth, height=0.24\textwidth]{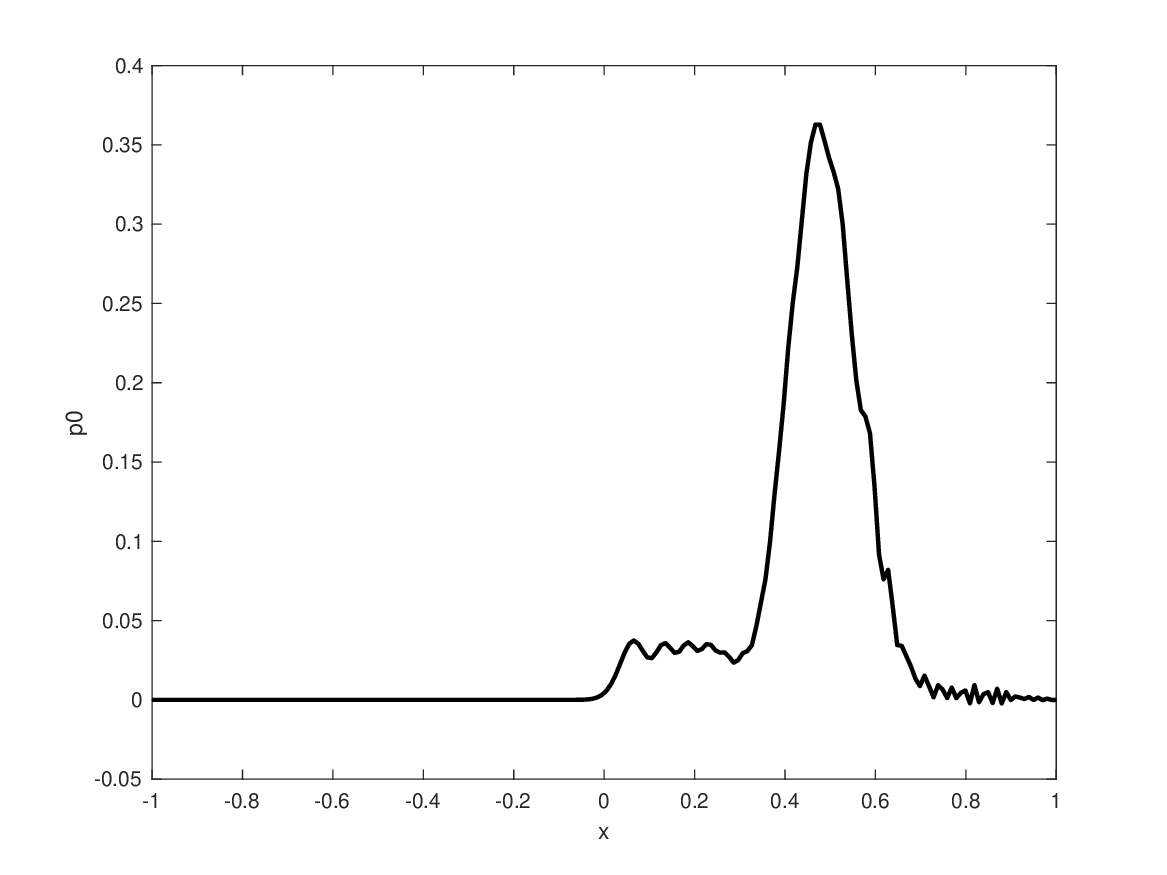} \label{fig:time1DG}}
\subfloat[CNN]{\includegraphics[width=0.24\textwidth, height=0.24\textwidth]{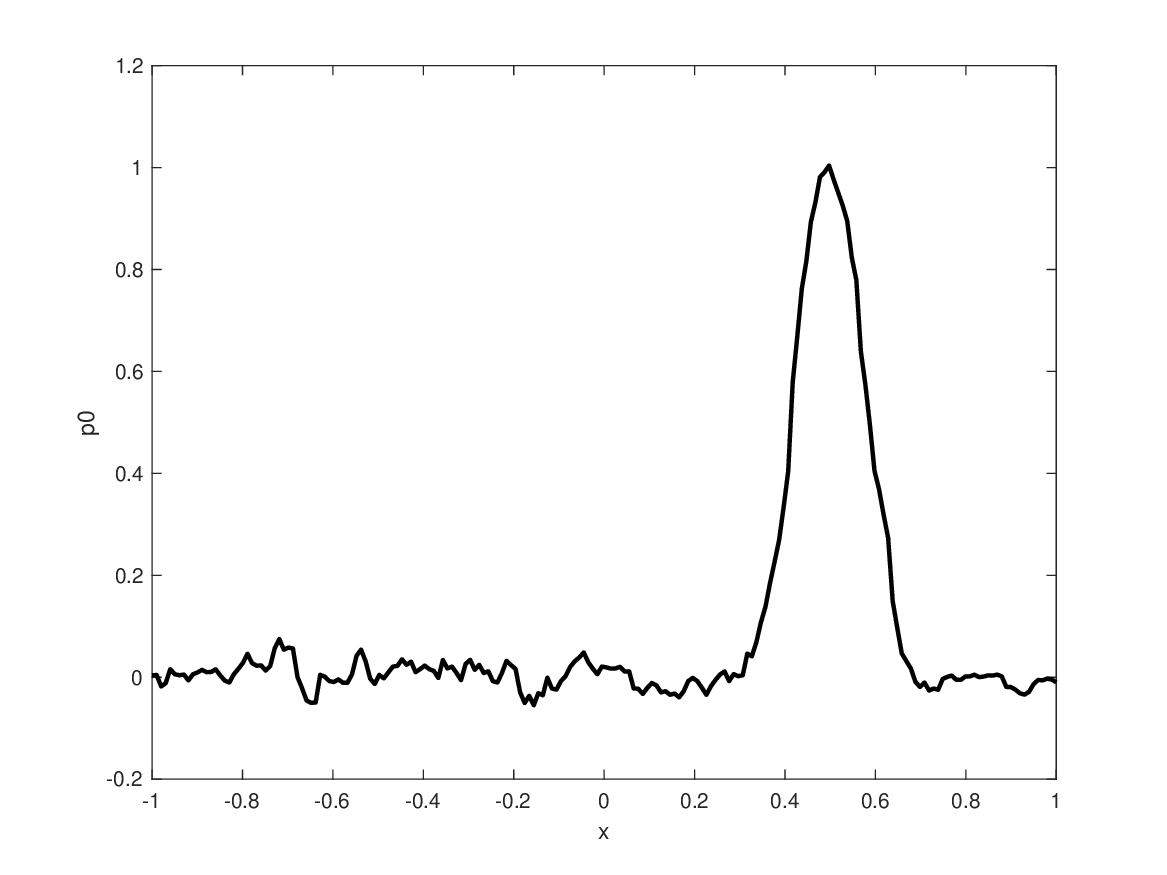} \label{fig:CNN1DG}}
\subfloat[SQH]{\includegraphics[width=0.24\textwidth, height=0.24\textwidth]{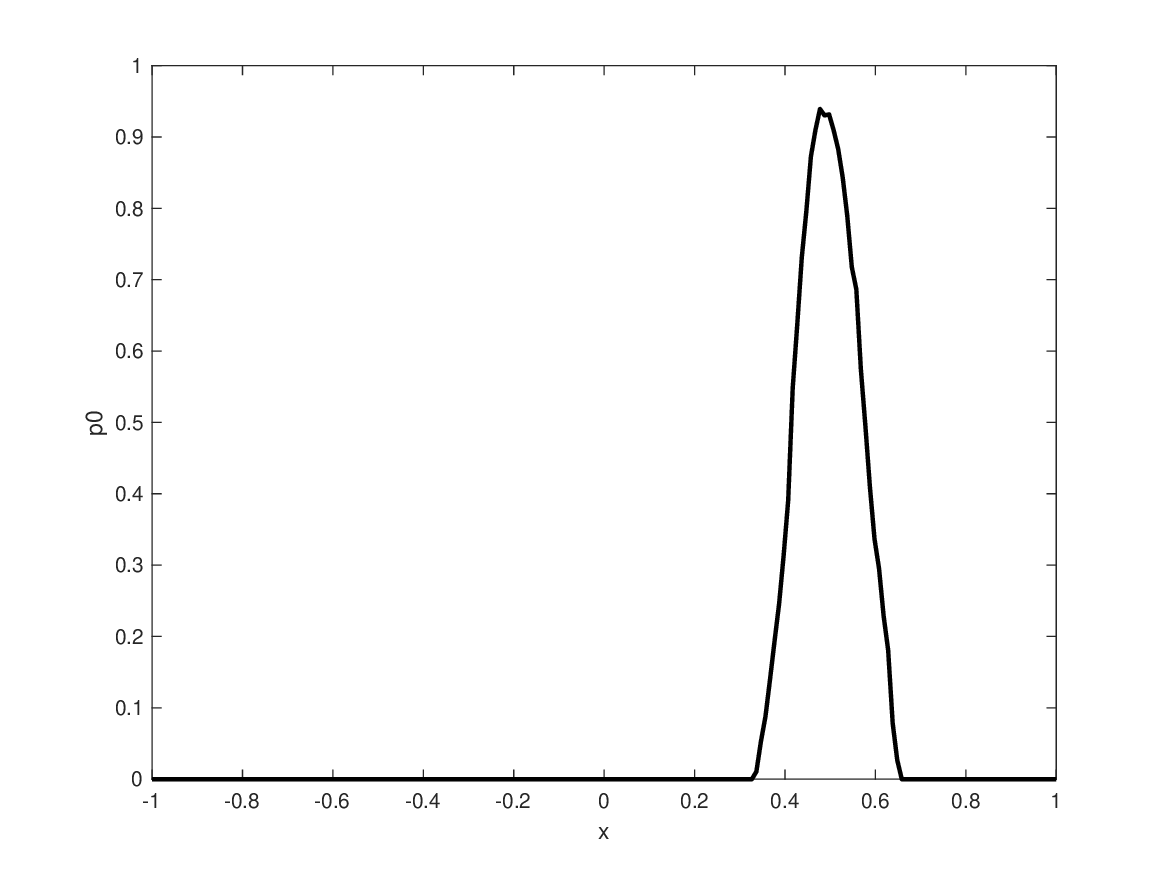}\label{fig:SQH1DG}}
\caption{Test Case 1: Reconstructions in 1D with the Gaussian phantom}
\label{fig:1DGaussian}
  \end{figure}
We observe that the time-reversal solution has significant loss of contrast due to damping and also has poor resolution. The reconstructions obtained using the CNN, improve the peak contrast significantly, but has poor background contrast. On the other hand, the reconstructions obtained using the SQH method have superior contrast and resolution. We note that the peak and the sparsity patterns in the reconstructions are well captured due to the presence of the $L^1$ regularization term.

For the second test case, we choose the true $p_0$ to be a characteristic function  centered at $x=-0.2$ with a value of 1.0 and width 0.3. The reconstructions are given in Figure \ref{fig:1Dchar}.

\begin{figure}[H]
\centering
\subfloat[True $p_0$]{\includegraphics[width=0.24\textwidth, height=0.24\textwidth]{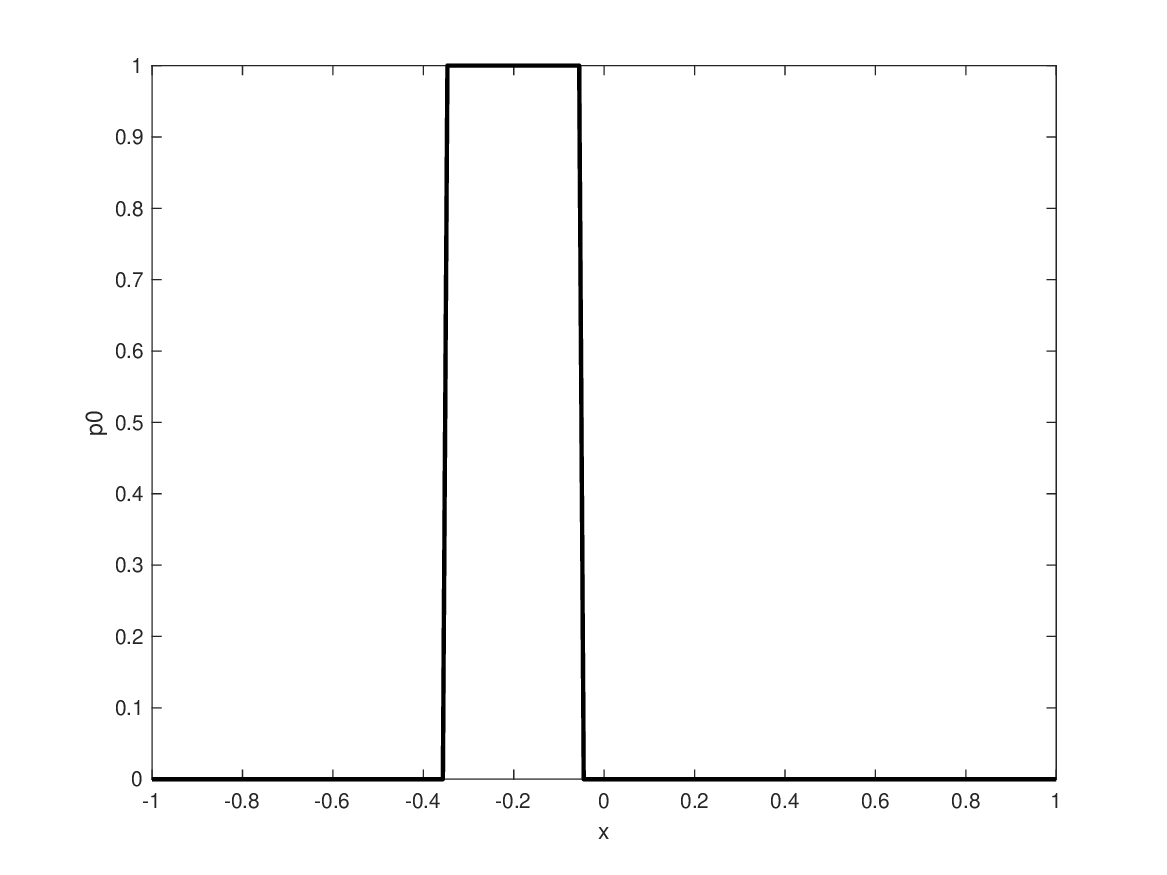} \label{fig:true1DC}}
\subfloat[Time-reversal]{\includegraphics[width=0.24\textwidth, height=0.24\textwidth]{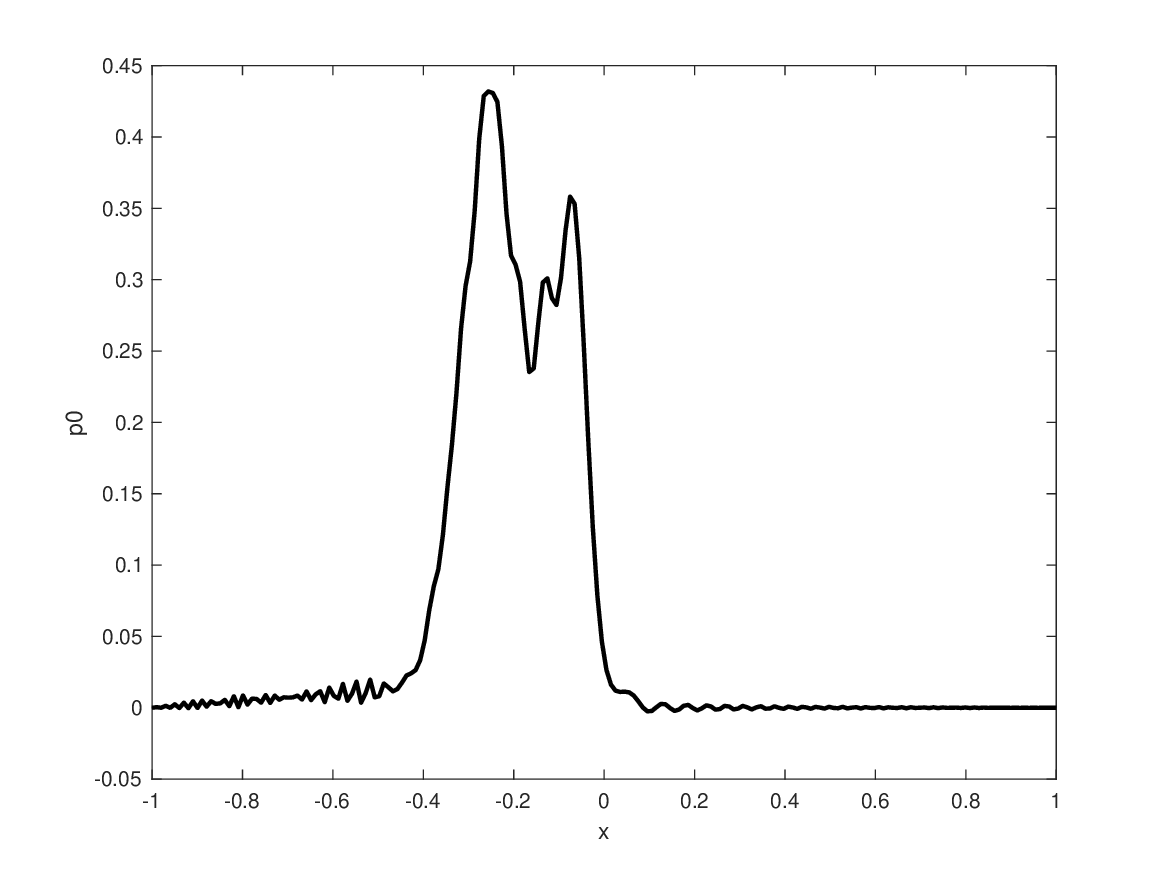} \label{fig:time1DC}}
\subfloat[CNN]{\includegraphics[width=0.24\textwidth, height=0.24\textwidth]{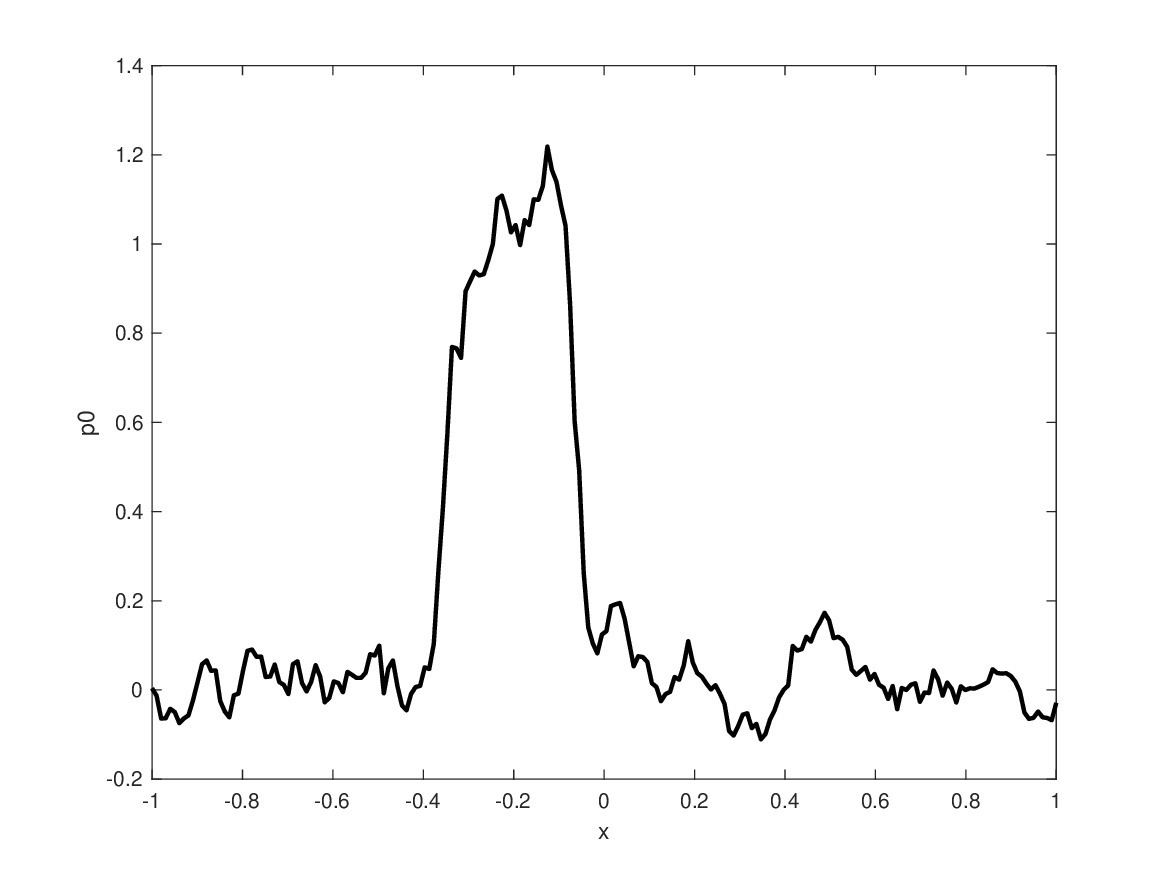} \label{fig:CNN1DC}}
\subfloat[SQH]{\includegraphics[width=0.24\textwidth, height=0.24\textwidth]{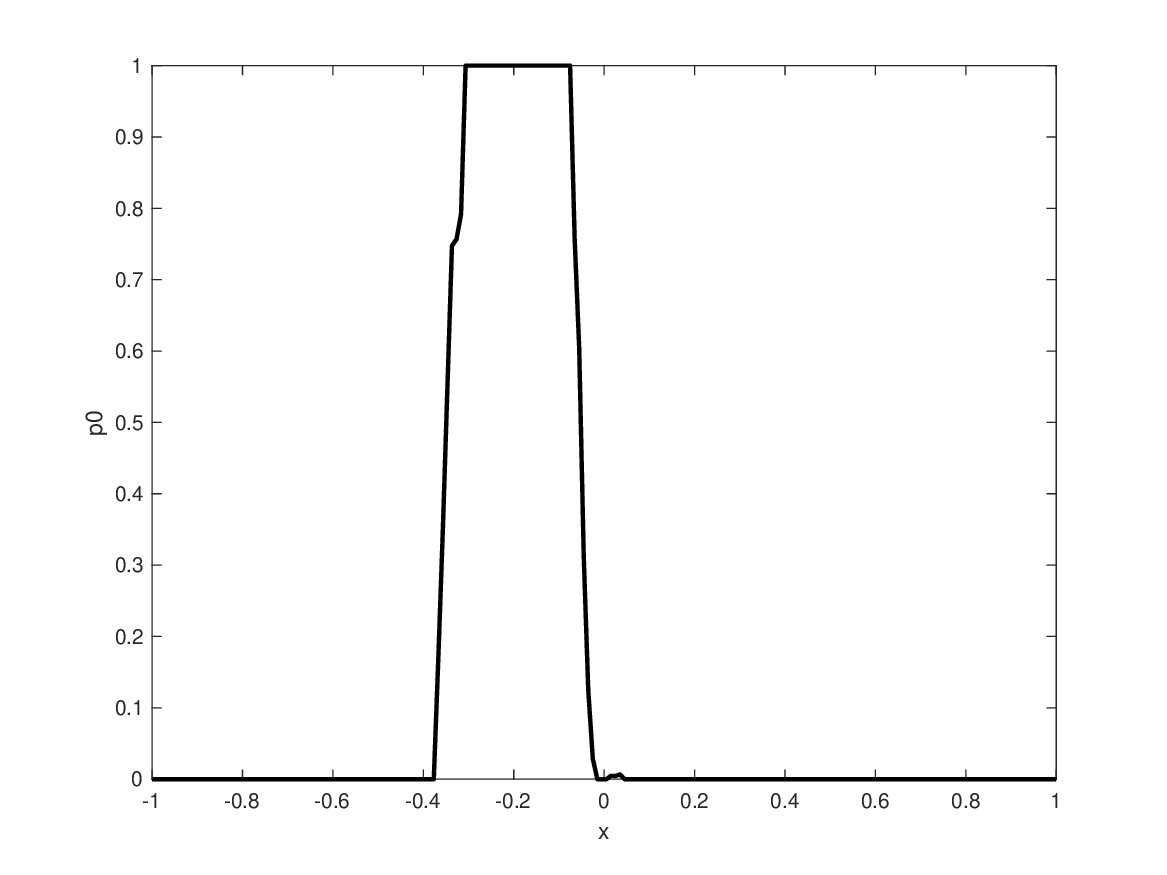}\label{fig:SQH1DC}}
\caption{Test Case 2: Reconstructions in 1D with the characteristic phantom}
\label{fig:1Dchar}
  \end{figure}
We now observe significant loss of contrast and resolution with the time-reversal reconstruction, with a dip along the place of the sharp-edge. With the CNN, the contrast improves again, but there seems to be appearance of spurious peaks, coupled with poor background contrast. On the other hand, with the SQH algorithm, we obtain a sharp-edge reconstruction, with a perfect background, demonstrating the superiority of our proposed method.

For the third test case, we choose the true $p_0$ to be a combination of a Gaussian  centered at $x=0.5$ with a peak value of 0.5 and standard deviation 0.25 and a characteristic function  centered at $x=-0.2$ with a value of 1.0 and width 0.2. The reconstructions are shown in Figure \ref{fig:1Dmixed}.

\begin{figure}[H]
\centering
\subfloat[True $p_0$]{\includegraphics[width=0.24\textwidth, height=0.24\textwidth]{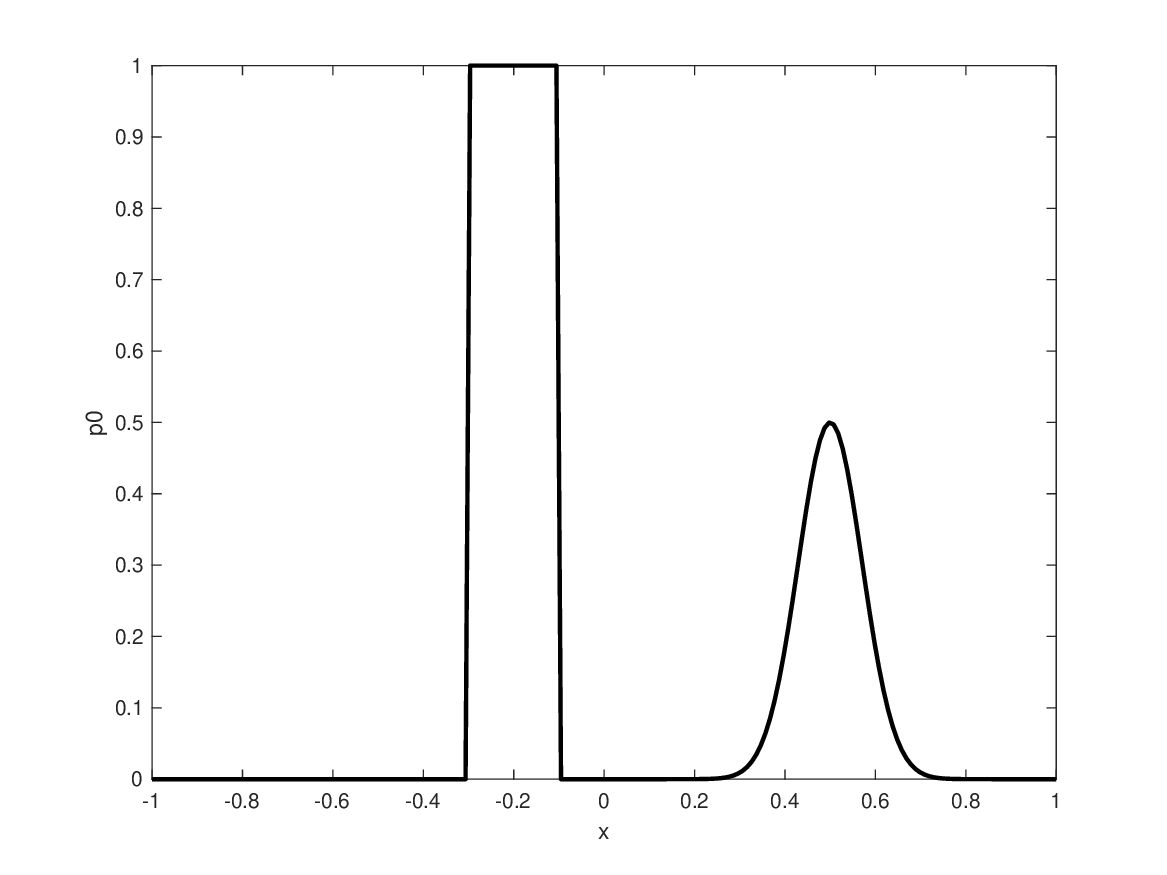} \label{fig:true1DM}}
\subfloat[Time-reversal]{\includegraphics[width=0.24\textwidth, height=0.24\textwidth]{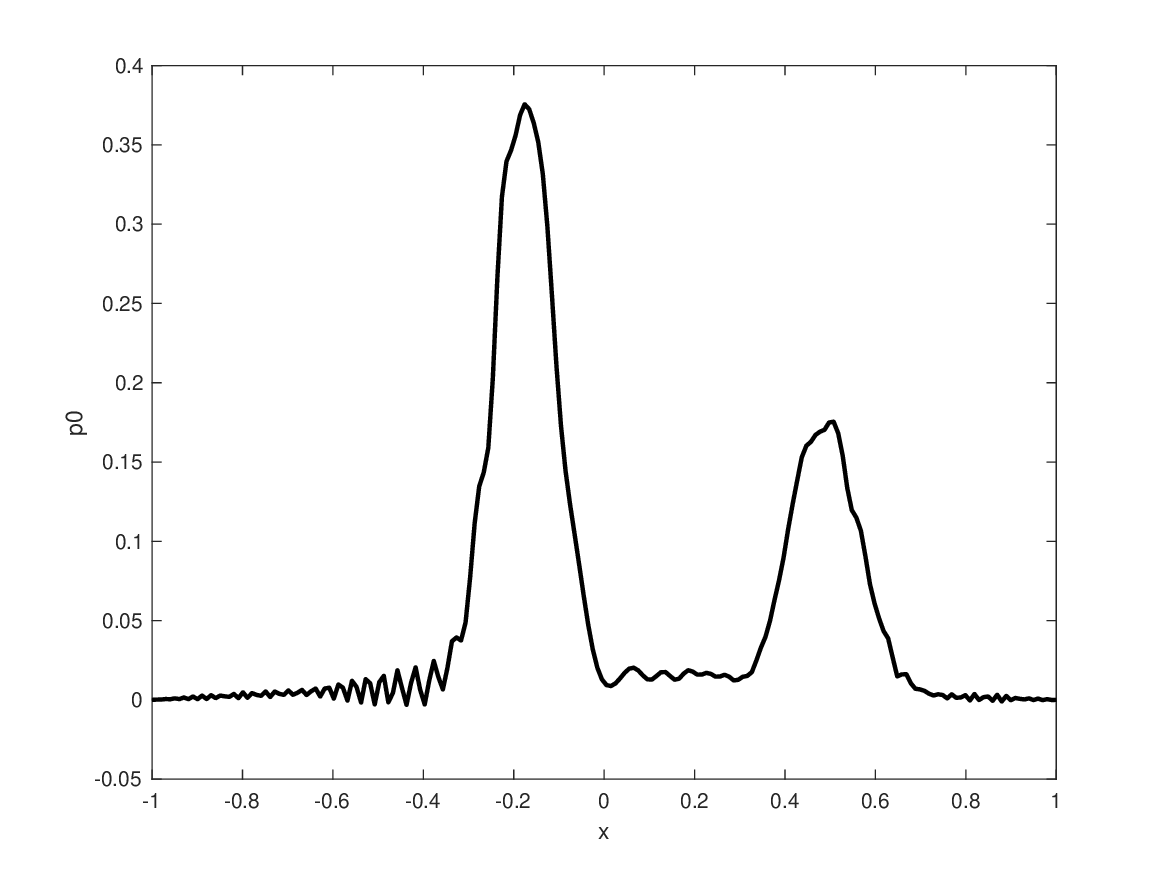} \label{fig:time1DM}}
\subfloat[CNN]{\includegraphics[width=0.24\textwidth, height=0.24\textwidth]{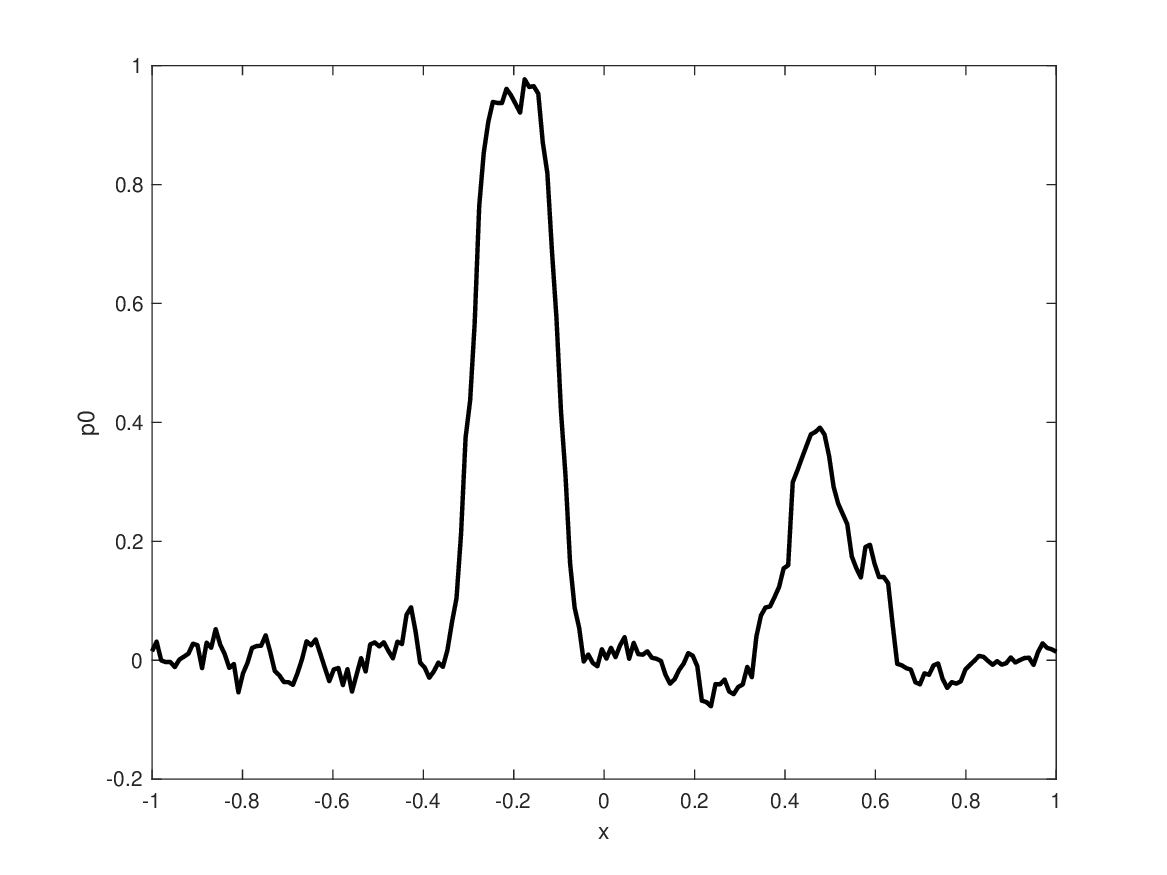} \label{fig:CNN1DM}}
\subfloat[SQH]{\includegraphics[width=0.24\textwidth, height=0.24\textwidth]{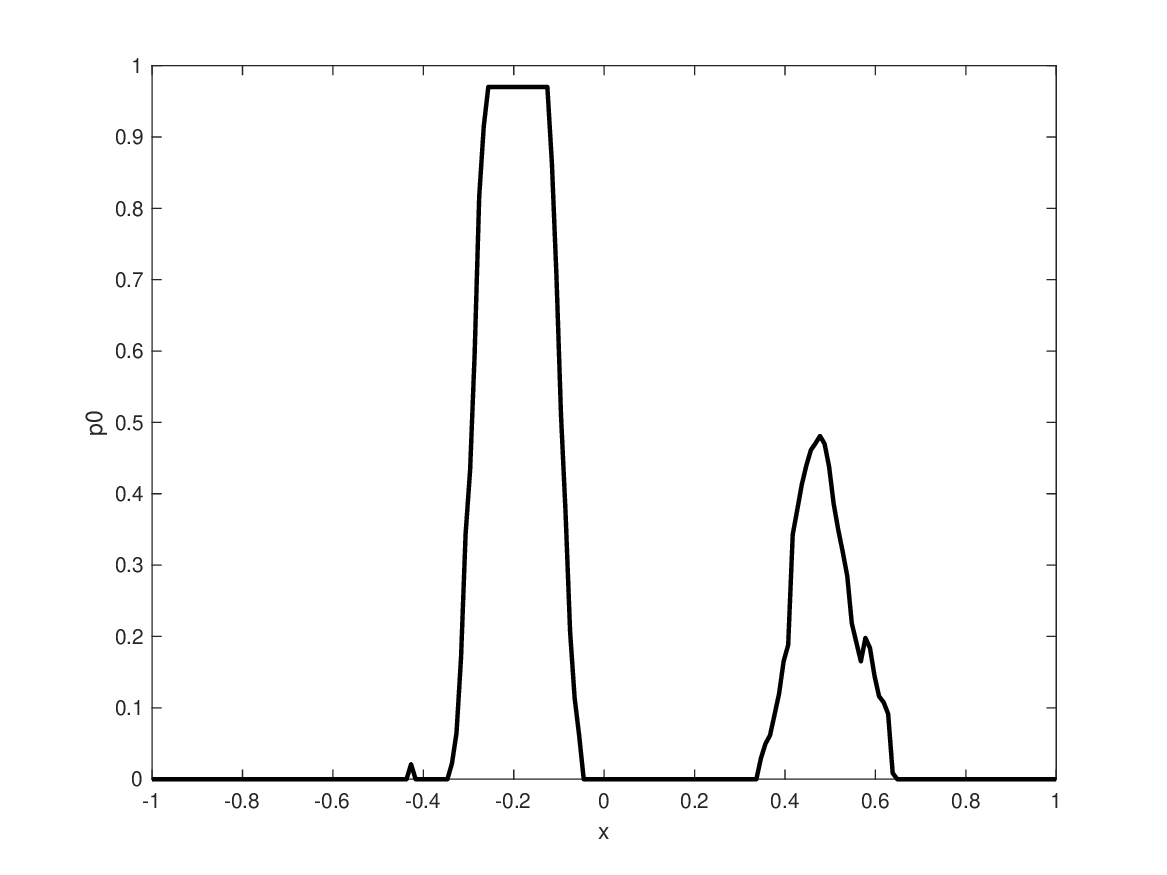}\label{fig:SQH1DM}}
\caption{Test Case 3: Reconstructions in 1D with the mixed phantom}
\label{fig:1Dmixed}
  \end{figure}
We again observe the poor reconstruction quality with the time-reversal algorithm, which improves in contrast with the CNN but still has poor background reconstruction. With the SQH, we again obtain the sharp-edge and sparsity patterns in the background, resulting in high resolution and contrast reconstruction. 

We now proceed to reconstructions in a two dimensional setup, where we choose our domain $\Omega = (-1,1)\times (-1,1)$, the observation boundary as $\partial\Omega$, and the final time of observation as $T=2$. We choose a non-trapping sound speed $c(\xx) = 1+w(\xx)*[0.1*\cos(2\pi x_1) + 0.05*\sin(2\pi x_2)],$ where $w(\xx)$ is a mollifier centered at the middle of the domain with radius $\sqrt{0.5}$, with a maximum value of 1. For the spatial grid, we choose 100 points along each dimension, whereas our time grid comprises of 200 points. To generate the data, we solve the wave equation \eqref{eq:attwave} in free space on a spatial grid with 50 points and on the temporal grid at 50 time points, and then interpolate the solution on the original grid to collect the data on the boundary $\partial\Omega$. 

For training the CNN, we generated a dataset of \(750\) samples: For the first 150 samples, the output was chosen as Gaussian functions with centers randomly drawn from the interval \( (-0.9, 0.9) \) and widths drawn from \((50, 150)\). For the next 350 samples, the output was chosen as characteristic functions with centers in the interval $(-0.9,0.9)$ and widths in the interval $(0.1,0.5)$. For the last 250 samples, we have a sum of a Gaussian and characteristic function as output, with centers of the Gaussian drawn from the interval $(-0.9,0.9)$ and widths in $(10,60)$ while the centers of the characteristic function are in $(-0.9,0.9)$ and widths in $(0.1,0.5)$. The number of epochs chosen for the training was 500, with batch size 32. The optimizer was chosen as ``Adam" with the Huber loss function.

In the next test case, our true phantom is a Gaussian phantom with center at $(-0.3,-0.3)$ and standard deviation 0.5. The reconstructions are shown in Figure \ref{fig:2Dgauss}.
\begin{figure}[H]
\centering
\subfloat[True]{\includegraphics[width=0.24\textwidth, height=0.24\textwidth]{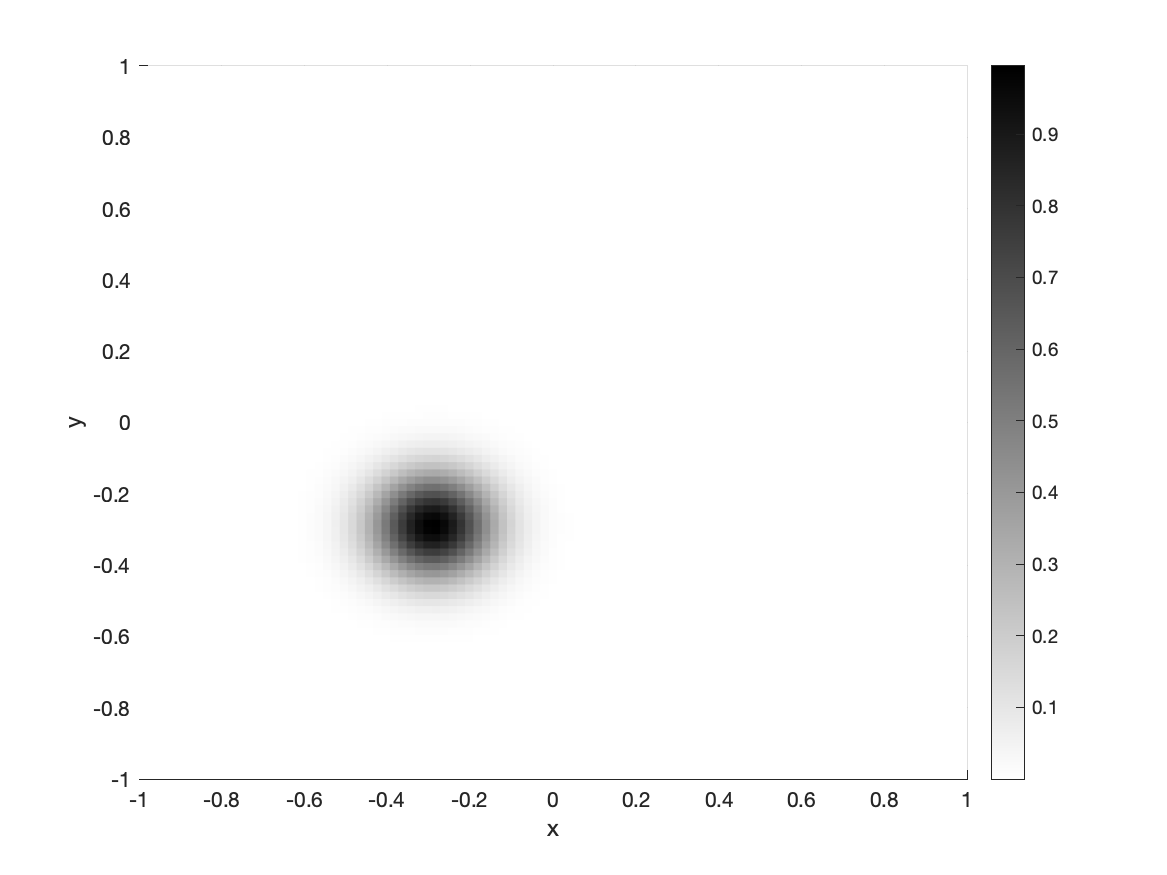} \label{fig:true2DG}}
\subfloat[Time-reversal]{\includegraphics[width=0.24\textwidth, height=0.24\textwidth]{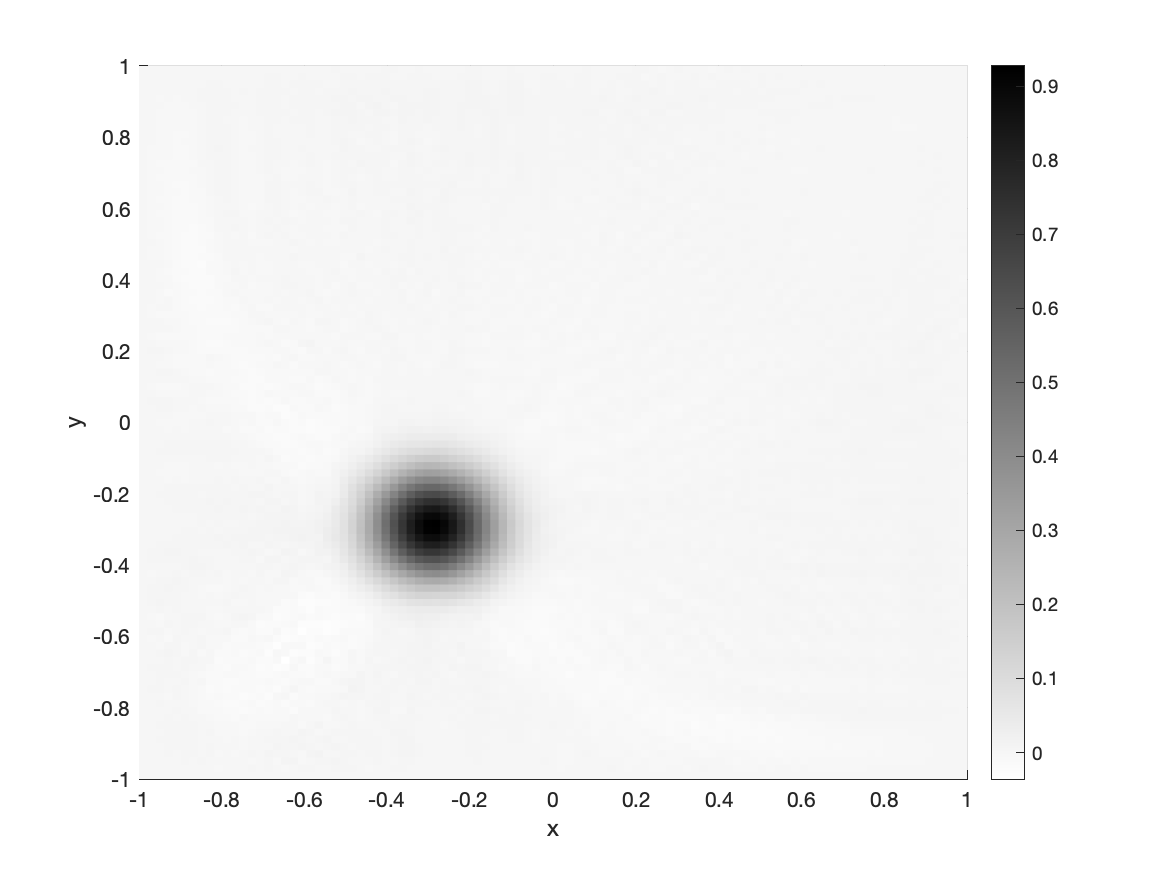}\label{fig:TR2DG}}
\subfloat[CNN]{\includegraphics[width=0.24\textwidth, height=0.24\textwidth]{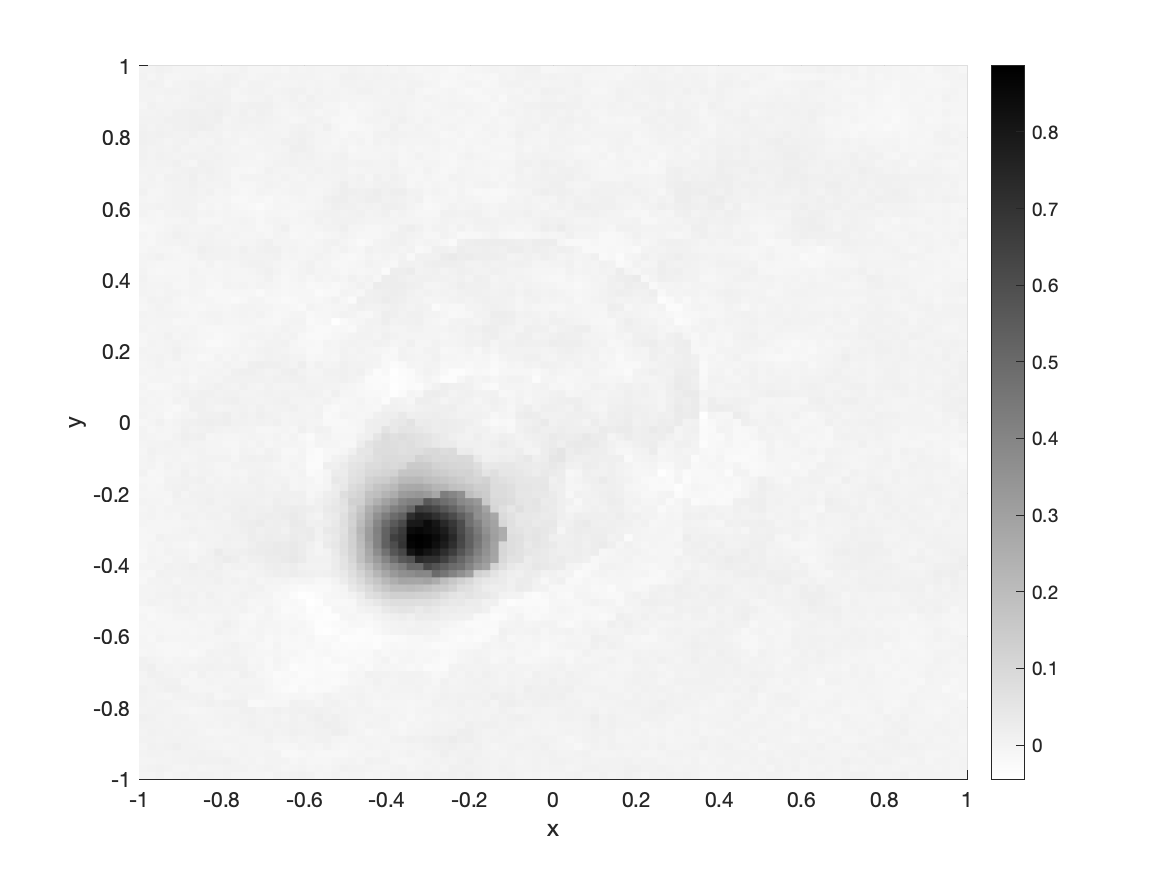}\label{fig:CNN2DG}}
\subfloat[SQH]{\includegraphics[width=0.24\textwidth, height=0.24\textwidth]{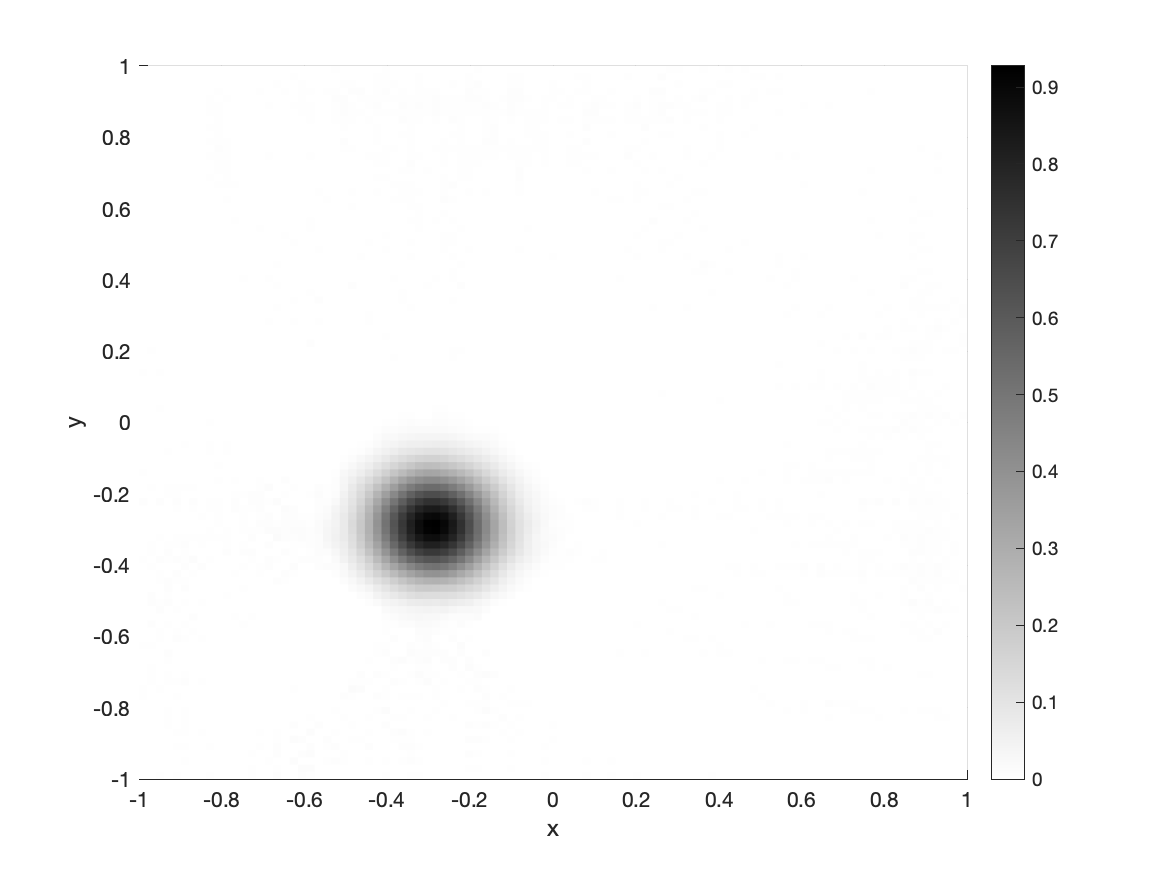}\label{fig:SQH2DG}}
\caption{Test Case 4: Reconstructions in 2D with a Gaussian phantom}
\label{fig:2Dgauss}
  \end{figure}
We observe that the time-reversal algorithm provides good resolution but not the best contrast. The CNN reconstruction is not the best in terms of resolution. On the other hand, the SQH algorithm provides superior resolution and contrast, outperforming both the methods. 

For the fifth test case, we now consider a disk phantom centered at $(-0.2,-0.2)$ with radius 0.1. The reconstructions are shown in Figure \ref{fig:2Ddisk}.
  \begin{figure}[H]
\centering
\subfloat[True]{\includegraphics[width=0.24\textwidth, height=0.24\textwidth]{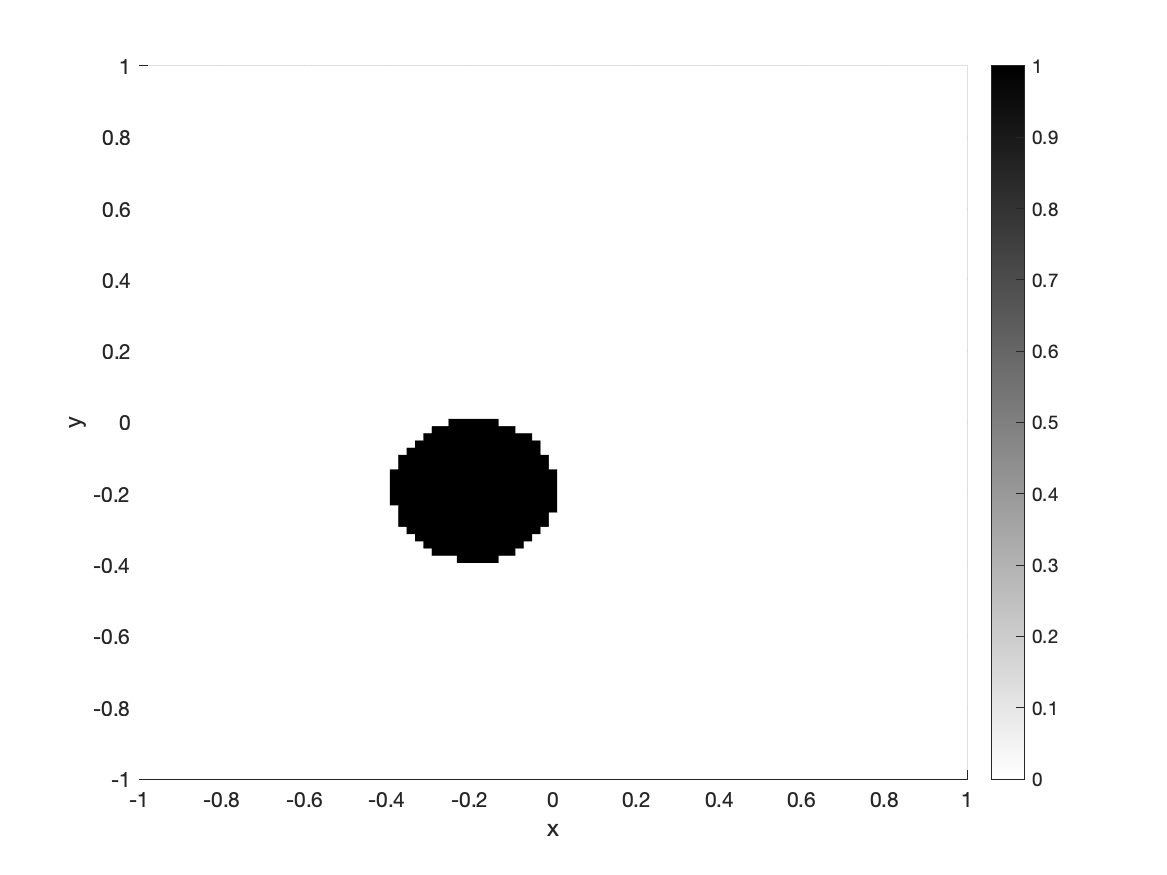} \label{fig:true2DC}}
\subfloat[Time-reversal]{\includegraphics[width=0.24\textwidth, height=0.24\textwidth]{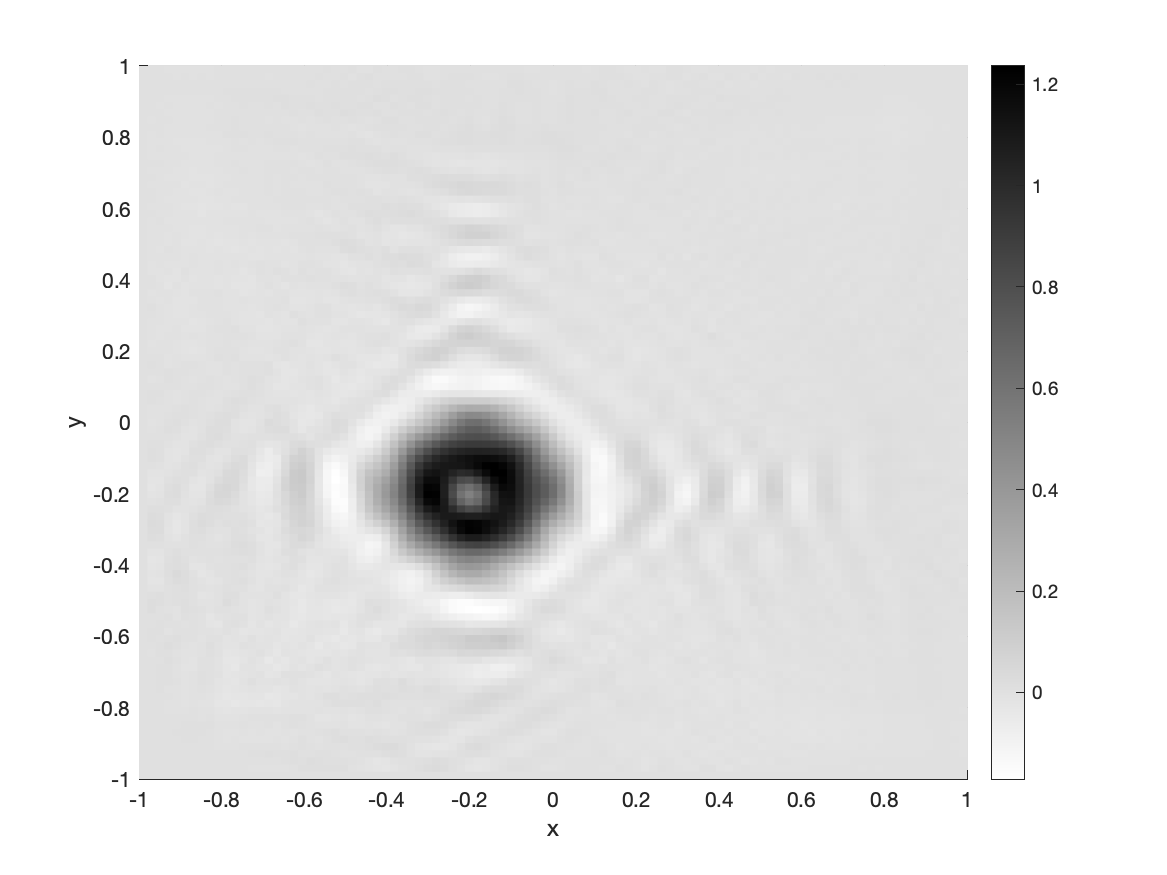}\label{fig:TR2DC}}
\subfloat[CNN]{\includegraphics[width=0.24\textwidth, height=0.24\textwidth]{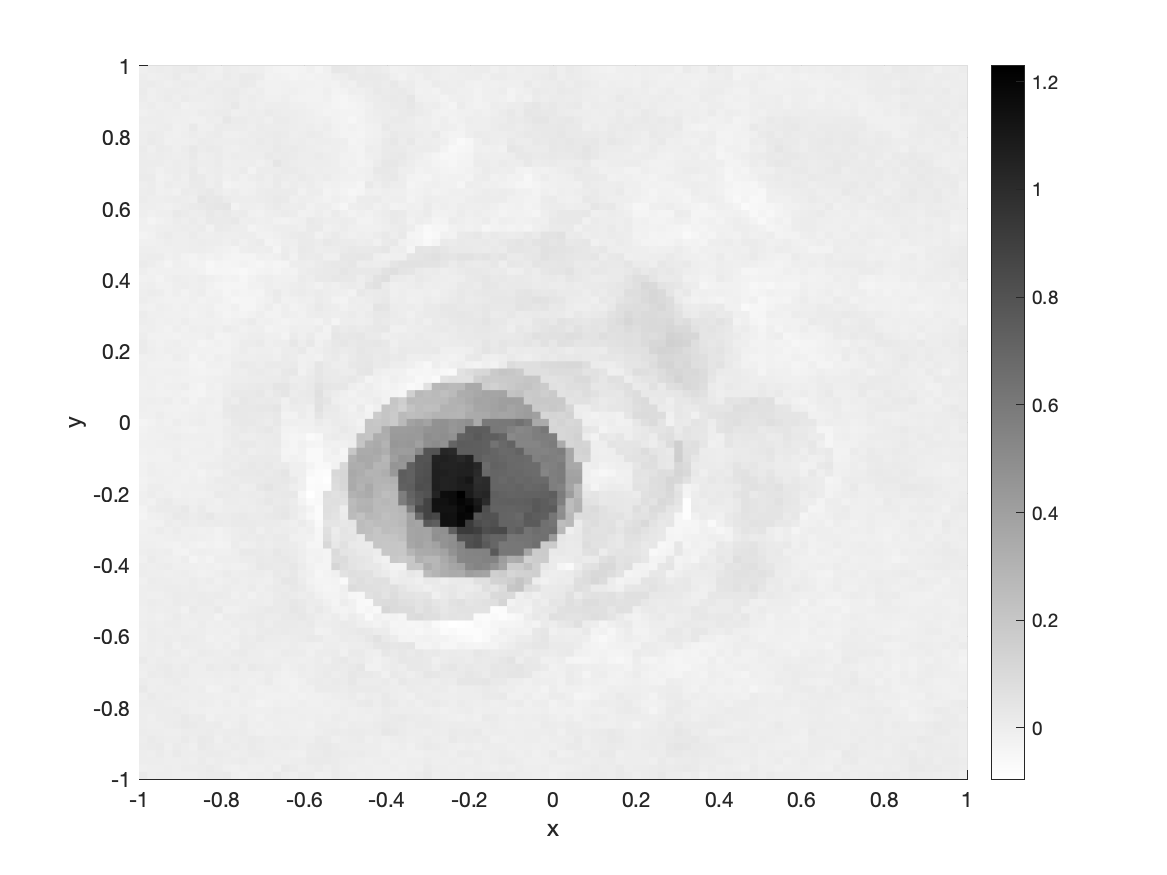}\label{fig:CNN2DC}}
\subfloat[SQH]{\includegraphics[width=0.24\textwidth, height=0.24\textwidth]{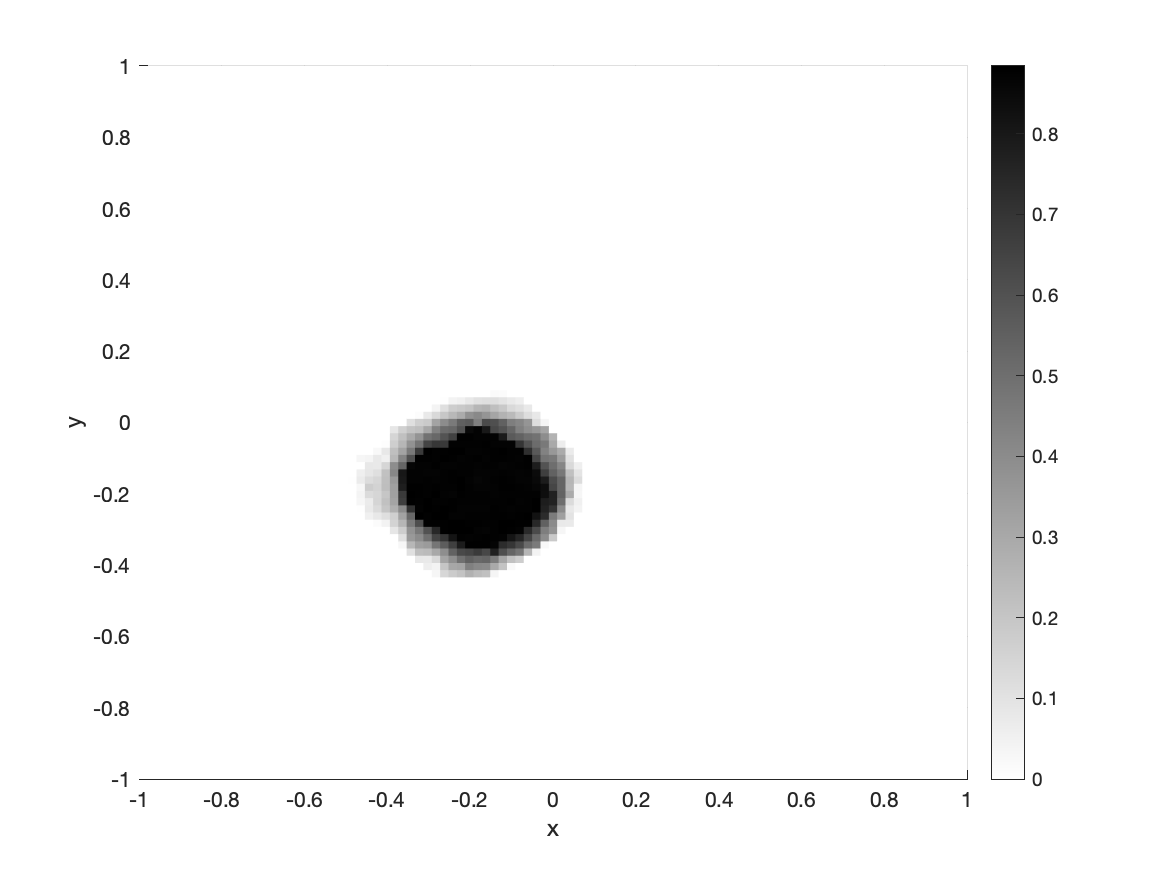}\label{fig:SQH2DC}}
\caption{Test Case 5: Reconstructions in 2D with a disk phantom}
\label{fig:2Ddisk}
  \end{figure}
We now notice the presence of severe artifacts in the reconstruction with the time-reversal algorithm, leading to poor resolution and contrast. The CNN reconstruction does capture the disk discontinuity with a better contrast, but has other superficial structures present. The SQH algorithm is able to significantly clear off the artifacts due to the sparsity promoting feature and the contrast and resolution are far more better. 
  \begin{figure}[H]
\centering
\subfloat[True]{\includegraphics[width=0.24\textwidth, height=0.24\textwidth]{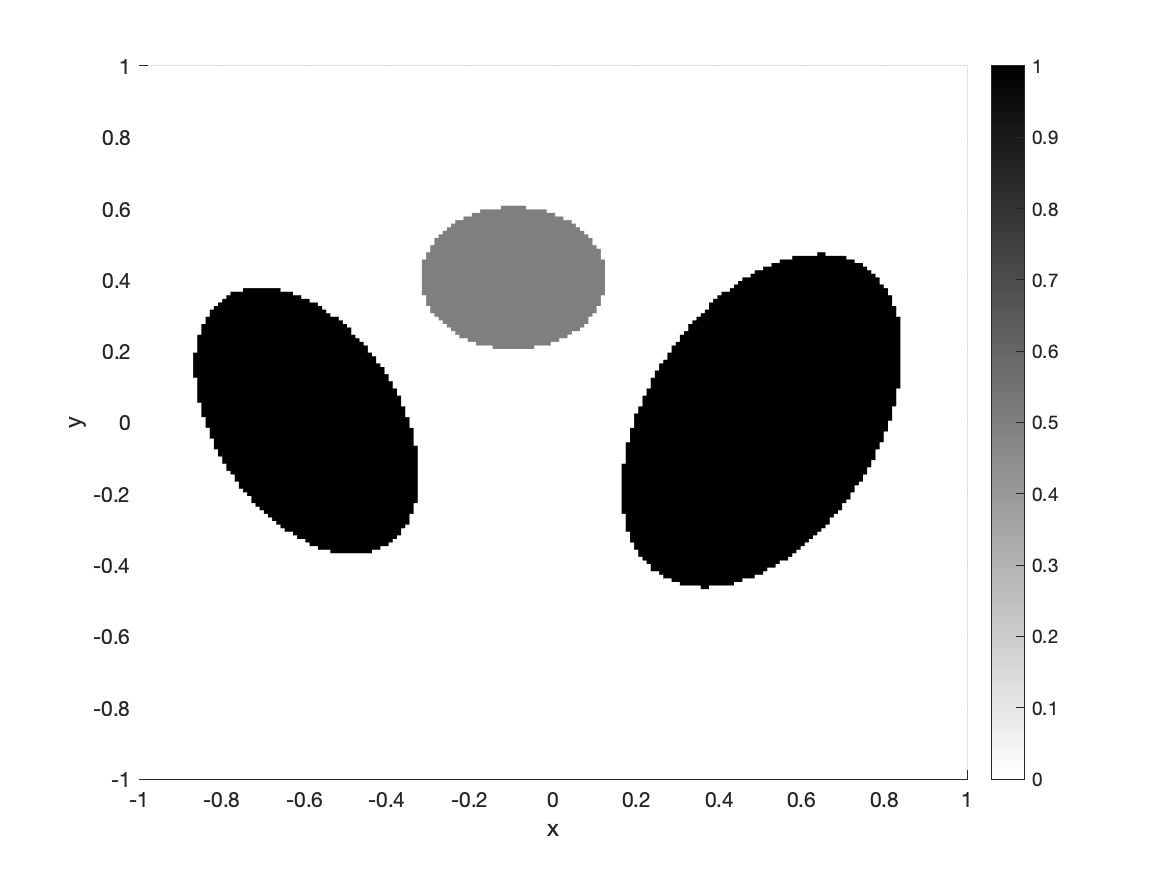} \label{fig:true2DH}}
\subfloat[Time-reversal]{\includegraphics[width=0.24\textwidth, height=0.24\textwidth]{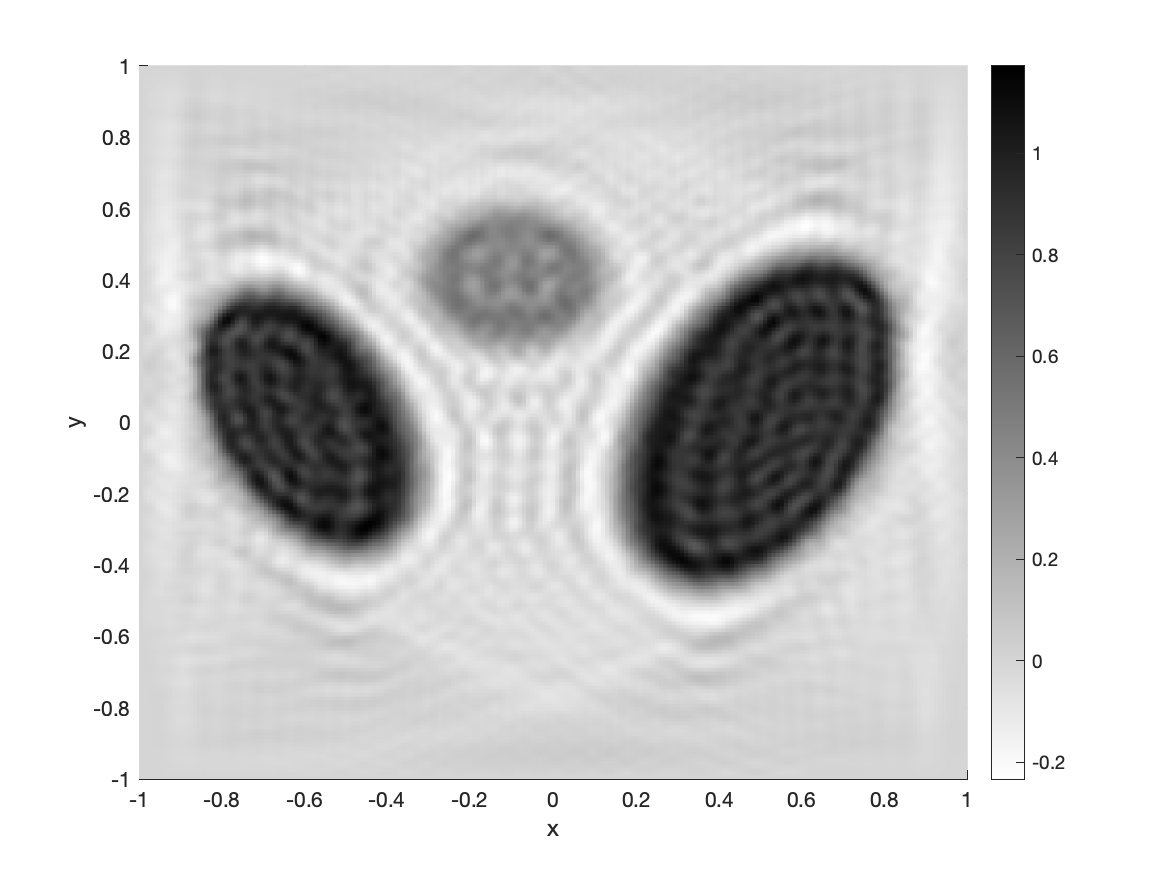}\label{fig:TR2DH}}
\subfloat[CNN]{\includegraphics[width=0.24\textwidth, height=0.24\textwidth]{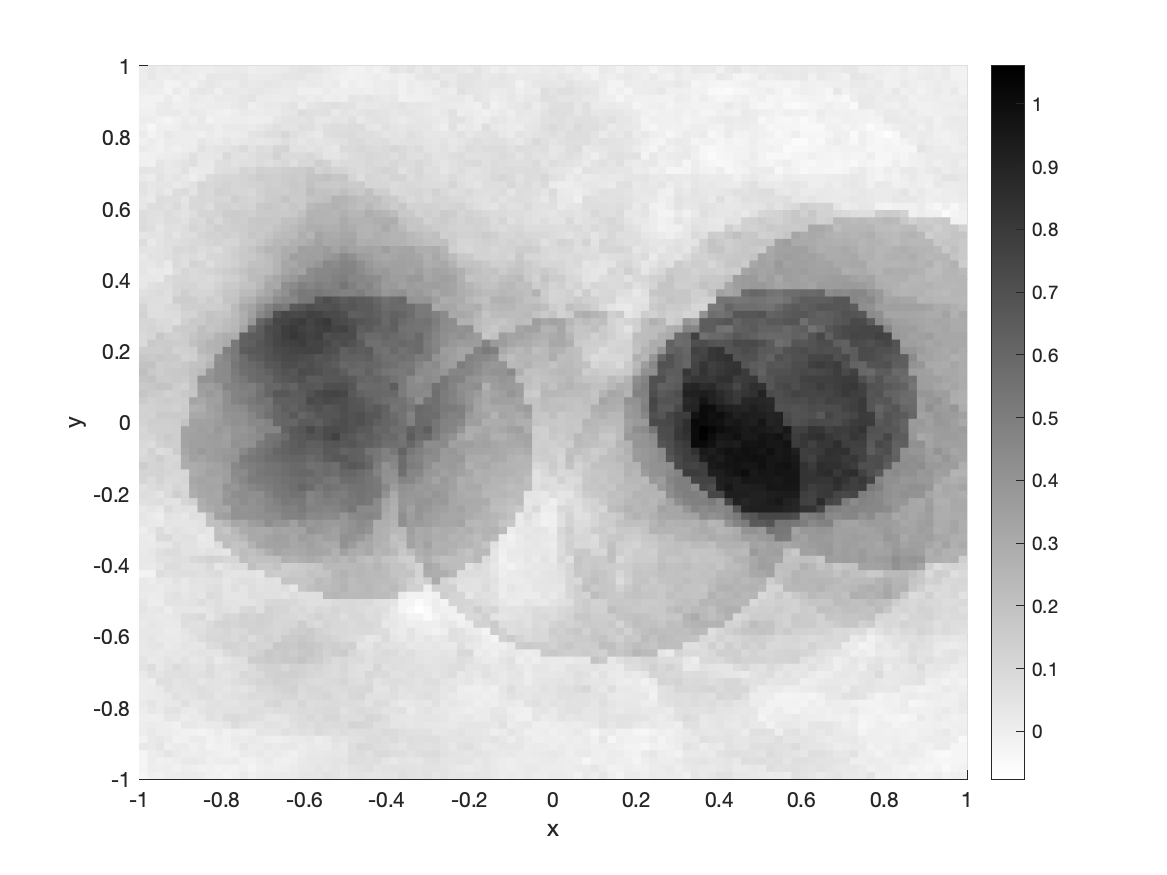}\label{fig:CNN2DH}}
\subfloat[SQH]{\includegraphics[width=0.24\textwidth, height=0.24\textwidth]{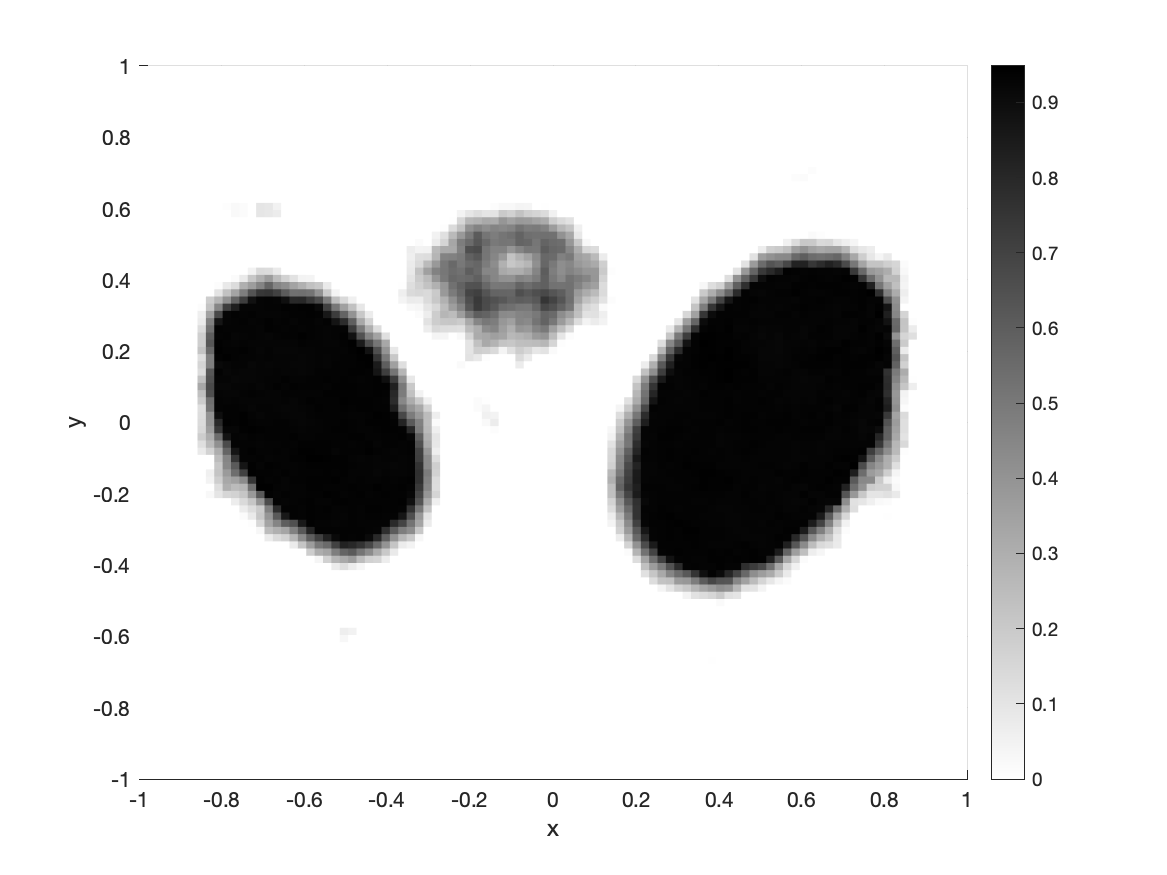}\label{fig:SQH2DH}}\\
\subfloat[True]{\includegraphics[width=0.24\textwidth, height=0.24\textwidth]{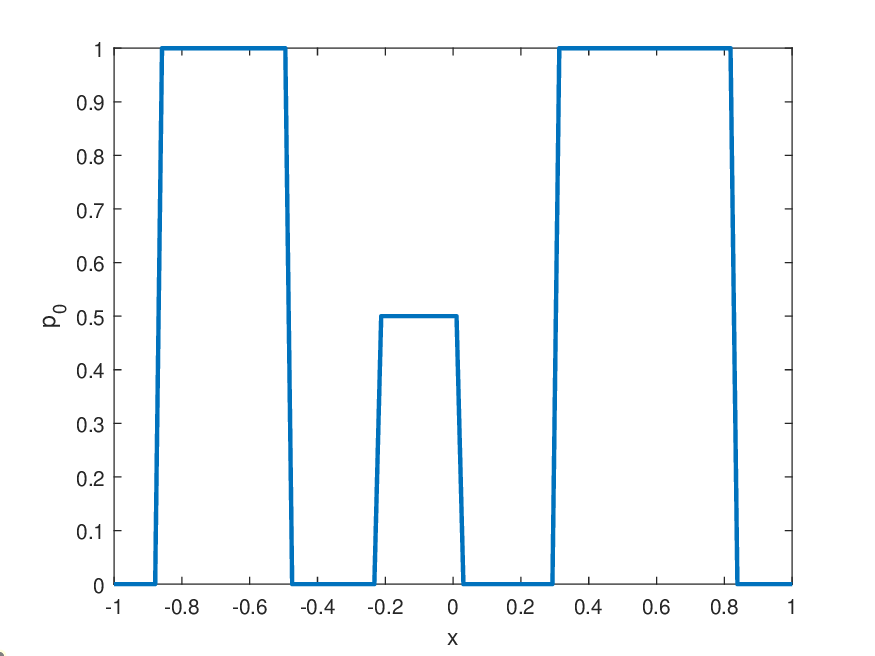} \label{fig:truecross2DH}}
\subfloat[Time-reversal]{\includegraphics[width=0.24\textwidth, height=0.24\textwidth]{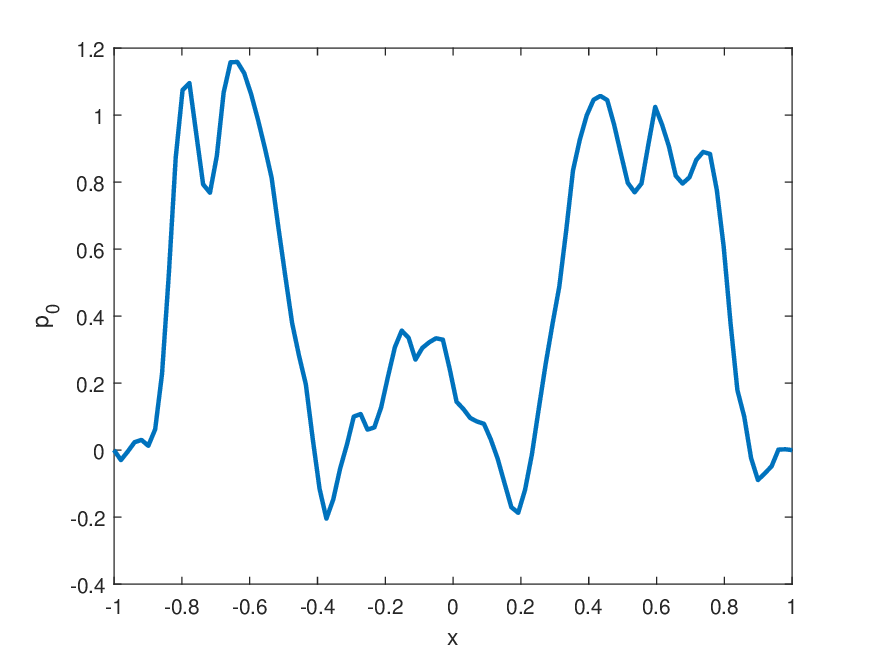}\label{fig:TRcross2DH}}
\subfloat[CNN]{\includegraphics[width=0.24\textwidth, height=0.24\textwidth]{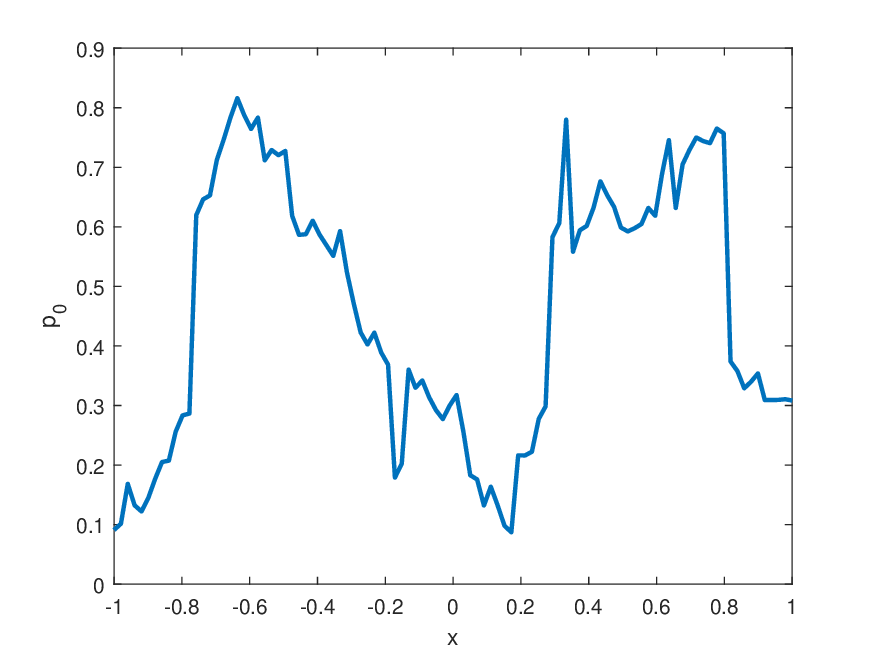}\label{fig:CNNcross2DH}}
\subfloat[SQH]{\includegraphics[width=0.24\textwidth, height=0.24\textwidth]{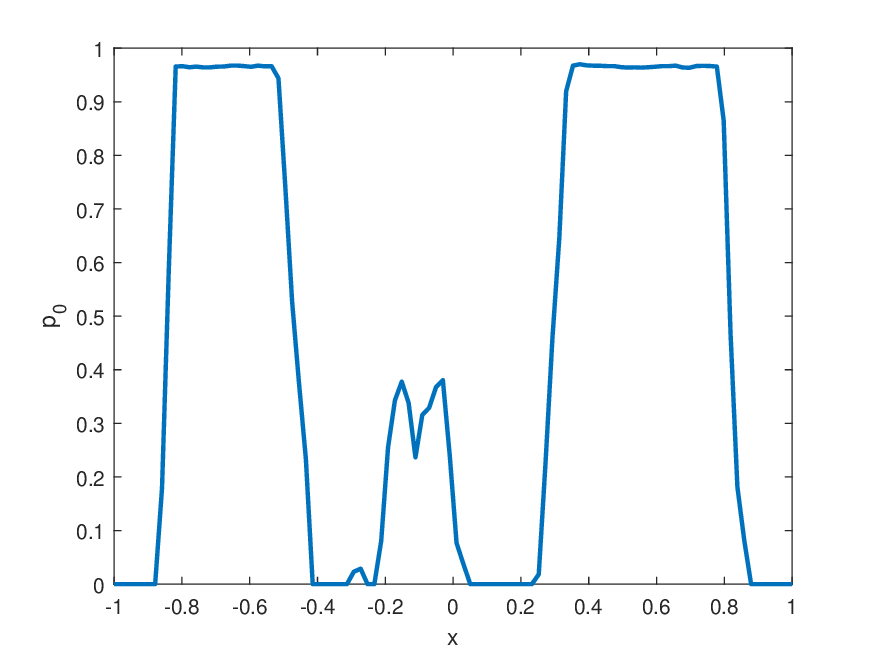}\label{fig:SQHcross2DH}}
\caption{Test Case 6: Reconstructions in 2D with the heart and lung phantom; Top row corresponds to the 2D profiles; Bottom row corresponds to a 1D cross sectional profile}
\label{fig:2Dheart}
  \end{figure}
 For the final test case, we consider a heart and lung phantom, represented by 2 ellipses and a disk. The reconstructions, with both 2D and cross-sectional 1D views, are shown in Figure \ref{fig:2Dheart}.
We again notice the presence of significant artifacts in the time-reversal reconstruction, leading to poor contrast and resolution. The CNN is expected to give a poor reconstruction with lot of superficial structures, since it was only trained on disks and Gaussians and at the most two objects, instead of ellipses and combination of three objects. However, some sharp edges in the CNN reconstruction help with the initial guess of the SQH algorithm, which now significantly improves the contrast and resolution significantly. This is also noted in the cross-sectional views of the reconstruction.

We now list the MSE, PSNR, and SSIM values of the reconstructions in Tables \ref{tab:MSE}, \ref{tab:PSNR}, and \ref{tab:SSIM}. 
%

\begin{table}[H]
\centering
\begin{tabular}{|c|c|c|c|c|c|}
\hline
Phantom & Test Case& TR & CNN &SQH \\ [0.5ex]
\hline
1D Gaussian &Test case 1  & 2.5e-2 & 8e-4 & 5e-4 \\
1D Characteristic function&Test case 2  &7.2e-2 & 1e-2 & 4e-3 \\
1D Mixed & Test case 3 & 6.1e-2 & 9.8e-2 &7.3e-3   \\
2D Gaussian &Test case 4  & 4.5e-5 & 6.4e-4 &2.9e-5 \\
2D Disk &Test case 5   & 4.3e-3 &8.4e-3 &3.8e-3  \\
2D Heart-Lung &Test case 6 & 1.8e-2 &7e-2 & 1.1e-2   \\
[1ex]
\hline
\end{tabular}
\caption{MSE values for the various test cases.}
\label{tab:MSE}
\end{table}

\begin{table}[H]
\centering
\begin{tabular}{|c|c|c|c|c|c|}
\hline
Phantom & Test Case& TR & CNN &SQH \\ [0.5ex]
\hline
1D Gaussian &Test case 1  & 15.94 & 31.22 & 32.74 \\
1D Characteristic function&Test case 2  & 11.46 & 19.94 & 23.71 \\
1D Mixed & Test case 3 & 12.13 & 20.10 & 21.34   \\
2D Gaussian &Test case 4  & 43.43 & 31.92 &45.30 \\
2D Disk &Test case 5   & 23.64 &20.78 &24.14  \\
2D Heart-Lung &Test case 6 & 17.41 &11.55 &19.55   \\
[1ex]
\hline
\end{tabular}
\caption{PSNR values for the various test cases.}
\label{tab:PSNR}
\end{table}

\begin{table}[H]
\centering
\begin{tabular}{|c|c|c|c|c|c|}
\hline
Phantom & Test Case& TR & CNN &SQH \\ [0.5ex]
\hline
1D Gaussian &Test case 1  & 0.77 & 0.47 & 0.94 \\
1D Characteristic function&Test case 2  & 0.71 & 0.2 & 0.92 \\
1D Mixed & Test case 3 & 0.59 & 0.35 & 0.95   \\
2D Gaussian &Test case 4  & 0.89 & 0.76 &0.96 \\
2D Disk &Test case 5   & 0.66 &0.47 &0.97  \\
2D Heart-Lung &Test case 6 & 0.28 &0.13 &0.84   \\
[1ex]
\hline
\end{tabular}
\caption{SSIM values for the various test cases.}
\label{tab:SSIM}
\end{table}
We observe that the TR method gives high values of MSE and low values of PSNR. The MSE decreases and PSNR increases with the CNN but is significantly outperformed by the SQH algorithm. The striking feature is the SSIM values of the SQH algorithm which is close to 1 for all the test cases compared to the other two methods, further demonstrating the robustness and versatility of our proposed reconstruction framework.
%

\section{Conclusions}
We have presented a theoretical framework for the recovery of the initial pressure of the attenuated wave equation from the boundary datum, using the idea of harmonic extension. When the damping is constant, we also obtained an explicit reconstruction formula for the initial pressure, given by a series representation. Furthermore, we present a gradient free optimization framework, in the realm of the Pontryagin's maximum principle, and develop the SQH method, guided by initial guess generated using a CNN, to provide superior reconstructions. Numerical experiments in one and two dimensions demonstrate the robustness and accuracy of our framework.

\section*{Acknowledgment}
S. Moon thanks The University of Texas at Arlington for its hospitality during 2025 when most of the presented work was done.

\bibliographystyle{plain}

\bibliography{s.moonref}
\end{document}